\documentclass[12pt,a4paper,reqno]{amsart}
\usepackage[utf8]{inputenc}
\usepackage[a4paper,margin=1in
]{geometry}
\usepackage[english]{babel}
\usepackage{float}

\usepackage[normalem]{ulem}
\usepackage{cancel}

\usepackage{verbatim}

\usepackage[colorinlistoftodos, textwidth=3cm]{todonotes}
\newcounter{dccomment}

\usepackage{latexsym}    
\usepackage{amssymb}   
\usepackage{amsmath}    
\usepackage{amsbsy}
\usepackage{amsthm}
\usepackage{amsgen}
\usepackage{amsfonts}
\usepackage{bbm}
\usepackage{array}
\usepackage{lmodern}

\usepackage{graphicx}
\usepackage{caption}
\usepackage[labelformat=simple]{subcaption}

\newsavebox{\largestimage}

\usepackage{tikz}
\usepackage{tikz-cd}
\usetikzlibrary{decorations.pathmorphing}
\usetikzlibrary{calc, decorations.markings}

\usepackage{xcolor}

\usepackage{verbatim}

\usepackage[shortlabels]{enumitem}

\usepackage{lineno}

\usepackage[nostamp]
{draftwatermark}
\SetWatermarkText{In Preparation}
\SetWatermarkScale{0.5}
\SetWatermarkColor[gray]{0.9}

\DeclareMathOperator{\Ch}{Ch}

\DeclareMathOperator{\kernel}{Ker}
\DeclareMathOperator{\Id}{Id}
\DeclareMathOperator{\proj}{proj} 
\DeclareMathOperator{\Sh}{Sh} 

\newcommand{\1}{\mathbbm{1}} 
\newcommand{\ar}{\mathrm{ar}}
\newcommand{\D}{\mathbb{D}} 

\newcommand{\fol}{\mathcal{F}} 
\newcommand{\G}{\mathcal{G}} 
\newcommand{\Hor}{\mathbb{H}} 
\newcommand{\hor}{\mathcal{H}}
\newcommand{\II}{\mathrm{II}} 
\newcommand{\Lie}{\mathcal{L}} 

\newcommand{\m}{\mathfrak{m}} 
\newcommand{\ob}{\mathrm{ob}}
\newcommand{\Or}{\mathcal{O}}

\newcommand{\R}{\mathbb{R}} 
\newcommand{\Sp}{\mathbf{S}} 
\newcommand{\Sec}{\mathrm{Sec}}
\newcommand{\V}{\mathbb{V}} 
\newcommand{\ver}{\mathcal{V}}

\newtheorem{theorem}{Theorem}  

	\newtheorem{thm}{Theorem}[section]
	
	\newtheorem{lemma}[thm]{Lemma}
	\newtheorem{cor}[thm]{Corollary}
	
	\newtheorem{proposition}[thm]{Proposition}

\theoremstyle{definition}	
	\newtheorem{remark}[thm]{Remark}
	\newtheorem{definition}[thm]{Definition}
	\newtheorem{example}[thm]{Example}
	
 \newtheoremstyle{TheoremNum}
        {\topsep}{\topsep}              
        {\itshape}                      
        {}                              
        {\bfseries}                     
        {.}                             
        { }                             
        {\thmname{#1}\thmnote{ \bfseries #3}}
    \theoremstyle{TheoremNum}

\usepackage[hyphens]{url}    
\usepackage{hyperref}
\hypersetup{
    colorlinks=true,
    linkcolor=blue,
    citecolor=blue,
    urlcolor=blue,
    pdfauthor={Diego Corro},
    pdftitle={Cheeger deformations for Lie groupoid actions}
}
\usepackage[hyphenbreaks]{breakurl}

\usepackage[reftex]{theoremref}

\title[Cheeger Deformations for Lie groupoid actions]{Cheeger Deformations for Lie groupoid actions}


\author[D.~Corro]{Diego Corro$^{\ast}$}
\address[D.~CORRO]{Fakultät für Mathematik\\ Karlsruher Institut für Tech\-no\-lo\-gie, Karlsruhe, Deutschland.}
\curraddr{School of Mathematics, Cardiff University, United Kingdom}
\urladdr{\url{www.diegocorro.com}}
\email{\href{mailto:diego.corro.math@gmail.com}{diego.corro.math@gmail.com}}
\thanks{$^\ast$Supported  DFG-Eigenestelle Fellowship CO 2359/1-1, and a UKRI Future Leaders
Fellowship [grant number MR/W01176X/1; PI: J Harvey].}

\begin{document}

\overfullrule = 5pt

\begin{abstract}
We give an extension of Cheeger's deformation techniques for smo\-oth Lie group actions on manifolds  to the setting of singular Riemannian foliations induced by Lie groupo\-ids actions. We give an explicit description of the sectional curvature of our generalized Cheeger deformation.	
\end{abstract}

\maketitle

\section{Introduction}


Given a  Riemannian manifold $(M,g)$ and an effective action by isometries of a Lie group $G$ on $M$, in in \cite{Cheeger1973} Cheeger presented a general deformation procedure with the aim of generalizing Berger's example of collapsing the unit round sphere $\Sp^3$ to a $2$-sphere $\Sp^2$ of radius $1/2$. The basic idea is to ``shrink'' the length of the vectors tangent to the orbits, while keeping the length of the vectors normal to the orbits unchanged (see  \cite{MoullieWebpage} for a nice visual explanation). This deformation procedure has been used to produce metrics with positive lower sectional or Ricci curvature bounds (for example \cite{LawsonYau1972}, \cite{SearleSolorzanoWilhelm2015}, \cite{Searle2023}, \cite{CavenaghiESilvaSperanca2023}).

Singular Riemannian foliations have been considered as a notion of symmetry since Lie group actions by isometries, as well as Riemannian submersions, are two large families that give examples of singular Riemannian foliations \cite{Alexandrino}, \cite{Corro2019}, \cite{CorroMoreno2020},  \cite{GalazGarcia2015}, \cite{GeRadeschi2013}, \cite{Molino}, \cite{Moreno2019}. But by \cite{Radeschi2014} there exist singular Riemannian foliations which are not induced by Riemannian submersions nor Lie group actions. 

In order to apply a Cheeger deformation, we need to know not only the equidistant decomposition of the manifold by the orbits of the action, but also the action itself. In the present work we define an extension of the Cheeger deformation procedure for singular Riemannian foliations induced by (proper) Lie groupoid actions. Lie groupoids have strong connections to regular foliations, e.g. \cite{Moerdijk}, \cite{Ehresmann1965}, \cite{Winkelnkemper1983}, and also more generally to singular  foliations in some cases, e.g. \cite{Debord2001}, \cite{AndroulidakisSkandalis2009}, and to closed singular Riemannian foliations, see \cite{AlexandrinoInagakiStruchiner2018}. It is of interest the question of whether it is possible if we can define a similar procedure to Cheeger deformations only knowing the geometric decomposition of the manifold, i.e. for a general singular Riemannian foliation. 

A Lie groupoid $\G\rightrightarrows M$  consists of a smooth manifold $\G$ of arrows,  a smooth manifold $M$ of points, source and target submersion maps $s,t\colon \G\to M$, and operations between the arrows analogous to the multiplication and inverse operations of a Lie group. Lie groupoids induce singular foliations on the base manifolds, where the leaves of the foliation $\fol_\G$ consists of the orbits $\G(p) = t(s^{-1}(p))\subset M$ of the groupoid.

Given a Lie groupoid $\G\rightrightarrows M$, $P$ a smooth manifold and $\alpha\colon P\to M$ a smooth map, we can define the notion of an action $(\G\rightrightarrows M)\overset{\mu}{\curvearrowright} (P\overset{\alpha}{\to} M)$ of the Lie groupoid $\G\rightrightarrows M$ on $P$ along $\alpha$. The Lie groupoid action consists of a new Lie groupoid $\G\times_M P\rightrightarrows P$ over $P$, where the space of arrows is the pullback of the submersion $s\colon \G\to M$ along $\alpha$. Thus a Lie groupoid action induces a singular foliation on $P$, where the leaves of the foliation are the orbits of the action groupoid. Moreover, when all the objects are compact, by \th\ref{T: existence of G->M invariant metrics on P->M} this singular foliation is a singular Riemannian foliation. 

In the present work we consider  proper Lie groupoid actions $\mu$ of $\G\rightrightarrows M$ along $\alpha\colon P\to M$, and show the existence of a family of Riemannian metrics $g_t$ on $P$ which are compatible with the orbits of the action, and such that the metrics ``shrink'' the orbits. Namely the foliation $(P,\fol,g_t)$ on $P$ is a singular Riemannian foliation (that is, the orbits are locally equidistant), where $\fol$ is the foliation induced by the orbits of the  action groupoid $\G\times_M P\rightrightarrows P$ induced by $\mu $ and $\alpha$.

\begin{theorem}\th\label{MT: deformation}
Let $\G\rightrightarrows M$ be a proper Lie groupoid acting along $\alpha\colon P\to M$ given by $\mu$, where $\G$, $M$ and $P$ are compact. Denote by $\fol$ the singular foliation on $P$ induced by the orbits of the action. Then for any $g$ transversely $\mu$-invariant Riemannian metric on $P$ (see \th\ref{D: mu-transversly invariant}) there exist a Riemannian metric $Q$ on $\G$ such that for the Riemannian metric $\widehat{g}_t = (\frac{1}{t}Q\oplus g)|_{\G\times_M P}$ we have a Riemannian submersion
\[
\bar{t}\colon (\G\times_M P,\widehat{g}_t)\to P.
\]
\end{theorem}

Thus the Riemannian metrics $\widehat{g}_t$ induce Riemannian metrics $g_t$ on $P$, making $\fol$ into a singular Riemannian foliation with respect to $g_t$ (see \cite{delHoyoFernandes2018}). Moreover, we show that we have a controlled description of the sectional curvature $K(g_t)$  of this family of metrics. Namely, applying O'Neill's formulas \cite{ONeill1966} we obtain the following result.

\begin{theorem}\th\label{MT: curvature of deformation}
Let $\G\rightrightarrows M$ be a Lie groupoid acting along $\alpha\colon P\to M$, where $\G$, $M$ and $P$ are compact. Denote by $\fol$ the singular foliation on $P$ induced by the orbits of the action. Then there exist Riemannian metrics $Q$ on $\G$, and $g$ on $P$ for which there exists a family of Riemannian metrics $\{g_t\}_{t\in (0,\infty)}$  on $P$ such that $(P,g_t)$ converges in the Gromov-Hausdorff sense to $(P/\fol,d_g^\ast)$ as $t\to \infty$, and 
\begin{linenomath}
\begin{align*}
K(g_t) = K(\widehat{g}_t)+3\|A(\cdot,\cdot)\|^2.
\end{align*}
\end{linenomath}
Here, $A$ denotes the the $A$-tensor of the submersion $\bar{t}\colon \G\times_M P\to P$ (\,see \cite{GromollWalschap}).
\end{theorem}

We point that applying the Gauss-Codazzi equations, we  can also compute the sectional curvature of $\widehat{g}_t$ via the embedding $(\G\times_M P,\widehat{g}_t)\hookrightarrow (\G\times P,(1/t)Q\oplus g)$  (see \th\ref{T: Proof of theorem B} for a full description of the sectional curvature of $g_t$).

This construction is analogous to the deformation procedure developed by Che\-eger for smooth compact Lie group actions on compact manifolds, first introduced in \cite{Cheeger1973}.

We also point out that when we consider a compact Lie group $\G = G\rightrightarrows \ast$ as a Lie groupoid and we consider the trivial map $\alpha\colon M\to \ast$, then a Lie groupoid action $(\G\rightrightarrows \ast) \curvearrowright (M\to \ast)$ is a smooth Lie group action $\G\times M\to M$. In this case we  have that $\G\times_{\ast} M = G\times M$, and we may take the metric $Q$ to be a bi-invariant Riemannian metric on $G$, and $g$ any $G$-invariant metric on $M$. In particular, since $G\times_\ast M = G\times M$, this implies that it is not necessary to compute the second fundamental form in \th\ref{MT: curvature of deformation}. In this case $K(\widehat{g}_t) = t^3K(Q)+K(g)$. Thus from, \th\ref{MT: deformation} and \th\ref{MT: curvature of deformation}, for the choices of $Q$ and $g$, we recover the classical construction of Cheeger deformations for group actions (see \cite{Mueter} and compare \th\ref{T: Classical cheeger deformation classical} to \th\ref{T: Proof of theorem B}). We also remark  that in the case  of a bi-invariant metric, the term $K(Q)$ corresponds to the sectional curvature of a bi-invariant metric, and thus for Lie group actions this deformation procedure is curvature non-decreasing. Nonetheless, in our construction we can consider a Riemannian metric $Q$ on $G$ which is not necessarily a bi-invariant metric; see \th\ref{R: 2-metrics on Lie groups} and \th\ref{T: Proof of theorem B} for example. 

We indicate that for a general Lie groupoid $\G$, the metric $Q$ in \th\ref{MT: deformation} is a so called $1$-metric given by a so called $2$-metric (see Section~\ref{S: Riemannian Lie groupoids}). From this fact, we point out that when when we consider $\G\rightrightarrows M$  to be the pair groupoid $M\times M\rightrightarrows M$, then for any choice of $\eta$ a Riemannian metric on $M$, the product metric $Q=\eta+\eta$ on $M\times M$ is a $1$-metric induced by a $2$-metric (see \th\ref{Ex: Pair groupoid as a Riemannian groupoid}). Thus, for a general groupoid, we cannot expect the term $K(Q)$ to be non-negative. Nonetheless, by \th\ref{T: Proof of theorem B} we see that for fixed tangent vectors in $P$, for sufficiently small $t$ the term $K(Q)$ becomes the dominating factor for determining the sectional curvature of the deformation $K(g_t)$.

As mentioned before, Cheeger deformations for Lie group actions have been applied in \cite{LawsonYau1972} to prove that a smooth action of a non-abelian compact Lie group  on a smooth manifold guarantees positive scalar curvature on a compact manifold. Moreover in \cite{SearleWilhelm2015} Cheeger deformations have been applied to lift curvature bounds from the orbit space $M/G$ of a compact Lie group action by isometries on a compact Riemannian manifold $(M,\eta)$. Also in \cite{CavenaghiESilvaSperanca2023} the authors give a proof of one of the main results in \cite{GroveZiller2002} using Cheeger deformations: given a compact Lie group $G$ acting smoothly on a compact smooth manifold $M$ such that $M/G$ homeomorphic to $[-1,1]$ and a principal orbit has finite fundamental group, then $M$ admits a $G$-invariant Riemannian metric with positive Ricci curvature. In all of these procedures it is crucially used the fact that $K(Q)$ is a non-negative or positive term, which may not be the case for Lie groupoid actions.

By \cite{AlexandrinoInagakiStruchiner2018} there is a strong connection between arbitrary closed singular Riemannian foliations and Lie groupoid actions. Namely, given $(M,\fol)$ a singular Riemannian foliation with closed leaves locally, then on sufficiently small tubular neighborhood $U$ of a fixed leaf $L$ there exists a foliation $\fol'\subset\fol$ which is induced by the orbits of an action Lie groupoid over the normal sphere bundle $\pi\colon U\to L$. Thus from \th\ref{MT: curvature of deformation} we  can change the metric in $U$ to have a local deformation of $(U,\fol')$ to $U/\fol'$. In some cases, such as orbit-like singular Riemannian foliations, or foliations of codimension $1$, we actually have that $\fol' = \fol$ over the set where $\fol'$ is defined. Nonetheless we point out that for an open cover of $M$ by such sufficiently small tubular neighborhoods there is no obvious way to glue such deformations to obtain a global deformation while retaining control over the sectional curvature. In \cite{Debord2001} sufficient conditions are given to determine when a general singular foliation, for which the set of regular leaves is dense, is given by the smallest possible holonomy Lie groupoid. But there are examples of singular Riemannian foliations for which such a groupoid does not exist (\cite[p. 496, Example 3]{Debord2001}). Thus the problem of giving a global deformation  procedure in an analogous sense to a Cheeger deformation, for a general closed singular Riemannian foliation remains open. 

Our article is organized as follows: In Section \ref{S: preliminaries} we review singular Riemannian foliations, Riemannian submersions and Cheeger deformations for compact Lie group actions by isometries. In Section \ref{S: Lie groupoid actions} we review the concepts of Lie groupoids, Lie groupoid actions and their connection to singular (Riemannian) foliations. In Section \ref{S: Cheeger like deformation for Lie groupoids actions} present our construction of a Cheeger like deformation for a Lie groupoid actions. We also present the proof  \th\ref{MT: deformation} in \th\ref{T: contraction of groupoid induces cheeger deformation}, and the proof of \th\ref{MT: curvature of deformation} in \th\ref{T: Cheeger deformation collapses} and \th\ref{T: Proof of theorem B}.

\section*{Acknowledgments}
I thank Marcos Alexandrino and Ilohann Sperança for useful conversations about Lie grou\-poids. I thank Marco Zambon for pointing me to the example of a the pair groupoid in the context of Riemannian submersions. I thank Fernando Galaz-García and John Harvey for useful comments that improved the presentation of this manuscript.

\section{Preliminaries}\label{S: preliminaries}

We start by presenting the classical theory of Cheeger deformations, and of Riemannian submersions.

\subsection{Riemannian submersions and curvature} A particular relevant example of a singular Riemannian foliation is a Riemannian submersion. Here we recall the relevant information about Riemannian submersions, but the interested reader may consult \cite{GromollWalschap}.

We consider $\pi\colon M\to B$ to be a smooth submersion from a smooth Rieamnnian manifold $(M,g)$ onto a smooth manifold $B$, i.e. for any $p\in M$ the differential $D_p\pi\colon T_pM\to T_{\pi(p)}B$ is a surjective map. We say that $\pi$ is a \emph{Riemannian submersion}, if for any vector field $X\in \kernel\, D\pi$ we have that $\mathcal{L}_X g=0$, i.e. the Lie derivative of $g$ in the direction of $X$ is zero. This implies that $g$ is constant along the fibers of $\pi$, and thus it induces a Riemannian foliation $\fol$ on $(M,g)$. Moreover, the Riemannian metric $g$ induces a smooth Riemannian metric $\tilde{h}$ on $B$. Given a submersion $\pi\colon (M,g)\to (B,h)$ between Riemannian manifolds we say that it is a Riemannian submersion if $h = \tilde{h}$. 

We denote by $\hor = \kernel\, D\pi^\perp \subset TM$ the \emph{horizontal distribution}, and by $\ver = \kernel\, D\pi$ the vertical distribution. Observe that $TM = \ver\oplus \hor$. Given a vector field $X\in \mathfrak{X}(M)$, we denote by  $X^v\in \ver$ the \emph{vertical component}, and by $X^h\in \hor$ the \emph{horizontal component}. In particular  a Riemannian submersion $\pi\colon (M,g)\to (B,h)$, is an isometry between the horizontal distribution and the tangent space of $B$. That is, we have have that $D\pi|_{\hor}\colon (\hor,g)\to (TB,\tilde{h})$ is a linear isometry on each fiber. 

Given a Riemannian submersion $\pi\colon (M,g)\to (B,h)$ we have the following relationships between the sectional curvature of the metrics $g$ and $h$.

\begin{thm}[O'Neill's formula \cite{ONeill1966}, see Section 1.5 in \cite{GromollWalschap}]\th\label{T: ONeills formula}
Let $\pi\colon (M,g)\to (B,h)$ be a Riemannian submersion. Then for any linearly independent $X,Y\in \mathfrak{X}(B)$ it holds
\begin{linenomath}
\begin{align*}
K_h(X,Y) =& K_g(\widetilde{X},\widetilde{Y})+3\Big\|A_{\widetilde{X}}\widetilde{Y}\Big\|^2_g\\
&=K_g(\widetilde{X},\widetilde{Y})+\frac{3}{4}\Big\|[\widetilde{X},\widetilde{Y}]^v\Big\|^2_g.
\end{align*}
\end{linenomath}
Here $\widetilde{X},\widetilde{Y}\in \hor$ are the unique horizontal lifts of $X,Y$, and $K_h(X,Y) = h(R_h(X,Y)Y,X)$ (analogous for $g$).
\end{thm}

\subsection{Cheeger Deformations}\label{S: Cheeger Deformations} Cheeger deformations where introdu\-ced by Che\-eger in \cite{Cheeger1973} as a way to describe several interesting results about the existence of metrics with lower curvature bounds. In this section we give a brief presentation of this construction, to familiarize the reader with them, and with the intention that the reader can contrast these deformations to the ones presented later in the manuscript.

Consider $M$ a  smooth manifold, and $H$ a compact Lie group. We say that \emph{$H$ acts smoothly on $M$} if there exists a smooth map $\mu\colon H\times M\to M$, such that
\begin{enumerate}[(i)]
    \item For $e\in H$ the identity element, and any $p\in M$ we have $\mu(e,p) = p$. 
    \item For any $h_1,h_2\in H$ and any $p\in M$ we have $\mu(h_1,\mu(h_2,p)) = \mu(h_1 h_2, p)$.
\end{enumerate}
The \emph{orbit through a point} $p\in M$ is the set $H(p) = \{\mu(h,p)\mid h\in H\}$. The \emph{isotropy subgroup} at $p\in M$ is the subgroup $H_p = \{h\in H\mid \mu(h,p) = p\}$. We say that the action is \emph{proper} if the map $\tilde{\mu}\colon H\times M\to M\times M$ given by $\tilde{\mu}(h,p) = (p,\mu(h,p))$ is a proper map. This means that the preimage of a compact set is compact. For a proper action, we have that for any point $p$ the subgroup $H_p$ is a closed subset. Moreover an orbits $H(p)$ is diffeomorphic to $H/H_p$.

Given a Riemannian metric $g$ on $M$, we say that it is $H$-invariant if for any $h\in H$ the map $\mu_h\colon (M,g) \to (M,g)$ defined as $\mu_h(p) = \mu(h,p)$, is an isometry with respect to $g$. In particular the partition $\fol_H = \{H(p)\mid p\in M\}$ is a singular Riemannian foliation of $(M,g)$. We refer to the leaf space $M/\fol_H = M/H$ as the \emph{orbit space of the action}. 

Given $\mu$ a smooth action of $H$ on a smooth manifold $M$, we can define a smooth action $\overline{\mu}$ of $H$ on the product $H\times M$, as follows:
\[
    \overline{\mu}(h_1,(h_2,p)) = (h_1 h_2,\mu(h_1,p)).
\]

Observe that the orbit space $(H\times M)/H$ is diffeomorphic to $M$ via the diffeomorphism $\tilde{\phi}\colon (H\times M)/H\to M$ given by 
\[
	\tilde{\phi}(H(h,p)) = \mu(h^{-1},p).
\]
If $(h',p')$ is an other representative of the orbit $H(h,p)$, we have that there exists $\tilde{h}\in H$, such that
\[
(\tilde{h}h,\mu(\tilde{h},p)) = \tilde{h}(h,p) = (h',p').
\]
Thus we have
\begin{linenomath}
\begin{align*}
	\mu (h',p') =& \mu((\tilde{h}h)^{-1},\mu(\tilde{h},p)) = \mu (h^{-1}\tilde{h}^{-1},\mu(\tilde{h},p)) = \mu(h^{-1}\tilde{h}^{-1}\tilde{h},p)\\
	 =& \mu(h^{-1},p).
\end{align*}
\end{linenomath}
Thus $\tilde{\phi}$ is well defined.

Assume that $H$ admits a bi-inviariant metric $Q$. We consider $(H\times M, \frac{1}{t} Q+g)$ and the map $t\colon H\times M\to M$ given by $t(h,p) = \mu(h^{-1},p)$. Observe that for each $t>0$ the metric  $\frac{1}{t}Q+g$ is $H$-invariant. Thus we obtain a family $\{g_t\}_{t>0}$ of Riemannian metrics on $M$, such that $t\colon (H\times M,\frac{1}{t}Q+g)\to (M,g_t)$ is a Riemannian submersion. We call $g_t$ the \emph{Cheeger deformation metric} of $g$ with respect to $Q$.

We denote by $e\in H$ the neutral element of $H$. Given a vector $x\in T_eH$ and $p\in M$, we define the \emph{action vector} $X^\ast(p)\in T_p M$ as
\[
D_{(e,p)}\mu(X,0)
\] 
In this way for each point $p\in M$ we can define a symmetric endomorphism $\Sh(p)\colon T_e H\to T_e H$, called the \emph{shape tensor} or \emph{orbit tensor}, via the identity:
\[
	Q(\Sh(p)(x),y) = g(X^\ast(p),Y^\ast(p))\quad \mbox{for all } y\in T_eH.
\]
For $t>0$ and $p\in M$ define the \emph{Cheeger tensor} $\Ch_t(p)\colon T_pM\to T_pM$ as follows. Given any vector $v\in T_p M$, there exists $x\in T_e H$ and $\xi\in \nu_p(H(p))$ such that $v= X^\ast(p)+\xi$. Then we set
\[
\Ch_t(p)(v) =\Big((\Id+t\Sh(p))^{-1}(x)\Big)^\ast(p)+\xi.
\]
This is an invertible tensor, and moreover by \cite[Satz 3.3]{Mueter} we have that for any $t\geq 0$ and $v,w\in T_p M$
\[
g_t(v,w) = g(\Ch_t(p)(v),w).
\]
In particular we get the following expression for the sectional curvature of $g_t$.

\begin{thm}[Section~3.d in \cite{Mueter}]\th\label{T: Classical cheeger deformation classical}
Consider $(M,g)$ a Riemannian, and $H$ a Lie group acting smoothly and effectively by isometries on $M$. We consider $Q$ a bi-invariant metric on $H$. Then for the Cheeger deformation metric $g_t$, we have for any $v=X^\ast(p)+\xi$, $w=Y^\ast(p)+\zeta\in T_pM$ the following
\begin{linenomath}
\begin{align*}
K(g_t)\Big(\Ch_t^{-1}(p)(v),\Ch^{-1}_t(p)(w)\Big) =& t^3 K(Q)(\Sh(p)(x),\Sh(p)(y))+ K(g)(v,w)\\
&+3\|A_{(t\Sh(p)(x),v)}(t\Sh(p)(y),w)\|^2_{\frac{1}{t}Q+g}\\
=& \frac{t^3}{4}\Big\|[\Sh(p)(x),\Sh(p)(y)]\Big\|^2_Q+K(g)(v,w)+\\ 
&+\frac{3}{4}\Big\|\big[(t\Sh(p)(x),v),(t\Sh(p)(y),w)\big]^v\Big\|^2_{\frac{1}{t}Q+g}.
\end{align*}
\end{linenomath}
\end{thm}

\begin{remark}
Observe that the last two terms
\[
k(Q)(\Sh(p)x,\Sh(p)y)=\frac{t^3}{4}\Big\|[\Sh(p)(x),\Sh(p)(y)]\Big\|^2_Q,
\]
and
\[
3\Big\|\big[(t\Sh(p)(x),v),(t\Sh(p)(y),w)\big]^v\Big\|^2_{\frac{1}{t}Q+g}
\]
are non-negative for $t\geq 0$. Moreover, by \cite[Section 3.d]{Mueter} for $w,v$ $g$-orthonormal we have 
\begin{linenomath}
\begin{align*}
\xi(t)&=\|\Ch_t^{-1}(p)(v)\|^2_{g_t}\|\Ch_t^{-1}(p)(w)\|^2_{g_t}-g_t(\Ch_t^{-1}(p)(v),\Ch_t^{-1}(p)(w))\\
&= t^2(\|\Sh(p)x\|^2_Q\|\Sh(p)y\|^2_Q-Q(\Sh(p)x,\Sh(p)y)^2)+t(\|\Sh(p)x\|^2_Q\|\Sh(p)y\|^2_Q)+1.
\end{align*}
\end{linenomath}
Since $\Sec(Q)(\Sh(p)x,\Sh(p)y)\geq 0$ and $\|[\Sh(p)x,\Sh(p)y]\|^2_Q$, the term 
\[
(\|\Sh(p)x\|^2_Q\|\Sh(p)y\|^2_Q-Q(\Sh(p)x,\Sh(p)y)^2)
\]
is non-negative. Thus when we have that in the case when $\Sec(g)\geq 0$, with a  as we take $t\to \infty$ we have $\Sec(g_t)\geq 0$ as $t\to \infty$.
\end{remark}

These facts have been exploited to improve the curvature of metrics in the presence of group actions. In particular when considering a lower bound for the Ricci curvature \cite{GroveZiller2002,SearleWilhelm2015}.

The classical example of a Cheeger deformation, are the so called Berger spheres \cite[Example~3.35]{CheegerEbin1975}. This is the unit round $3$-sphere equipped with the Hopf fibration action by $\Sp^1$, and the family of metrics obtained from deforming the standard round  metric given by the embedding $\Sp^3\hookrightarrow \R^3$ via the action. For this example, the sectional curvature remain bounded, and at the limit as $t\to 0$, one obtains a sequence converging to the $2$-sphere with a constant metric of curvature $4$ \cite{Weber2009}. 

\section{Lie groupoid actions}\label{S: Lie groupoid actions}

In this section we present the concept of a Lie groupoid, and their connection to singular Riemannian foliations. The interested reader can consult \cite{GarmendiaGonzales2019} and \cite{Wang2018} for a more complete exposition.

A \emph{Lie groupoid $\G\rightrightarrows M$} consists of two smooth manifolds $\G$, referred to as the \emph{set of arrows}, and $M$, which is called the \emph{set of objects}, together with the following set of smooth maps:

\begin{enumerate}[a)]
 \setlength\itemsep{0.3em}
\item The \emph{source map} $s\colon \G\to M$, and the \emph{target map} $t\colon \G\to M$, which are surjective submersions.
\item The \emph{unit map} $\1\colon M\to \G$, which is an embedding. We refer to $\1_p = \1(p)$ as an \emph{unit element}.
\item The \emph{inversion map} $i\colon \G\to \G$, which is a diffeomorphism.
\item The set of \emph{composable arrows} $\G^{(2)} = \{(g,h)\in \G\times \G\mid s(g) = t(h)\}$ and the \emph{multiplication map} $\m\colon \G^{(2)}\to \G$.
\end{enumerate}
These maps satisfy the following conditions:
\begin{enumerate}[i)]
\item The source map is \emph{inviariant under left multiplication}, i.e. for all $(g,h)\in \G^{(2)}$ we have
\[
	s(\m(g,h)) = s(h).
\]
\item The target map is \emph{inviariant under right multiplication}, i.e. for all $(g,h)\in \G^{(2)}$ we have
\[
	t(\m(g,h)) = t(g).
\]
\item The multiplication is \emph{associative}, i.e. for all $(\tilde{g},g),(g,h)\in \G^{(2)}$ we have
\[
	\m(\m(\tilde{g},g),h) = \m(\tilde{g},\m(g,h)).
\]
\item The units act trivially: i.e. for all $p\in M$ and all $g\in \G$ we have
\[
	s(\1_p) = p = t(\1_p), \qquad \m(\1_{t(g)},g) = g = \m(g,\1_{s(g)}).
\]
\item The inversion map acts as an \emph{inversion} for the multiplication: i.e. for all $g\in \G$ we have
\[
	\m(g,i(g)) = \1_{t(g)},\qquad \m(i(g),g) = \1_{s(g)}.
\]
\end{enumerate}
The collection $\{s,t,\1,i,\m\}$ are called the \emph{structure maps} of the Lie groupoid $\G\rightrightarrows M$. 

The set $L_p = t(s^{-1}(p))$ is called the \emph{orbit through $p$} and it is an embedded submanifold of $M$. The \emph{isotropy group at $p$} is defined $\G_p = s^{-1}(p)\cap t^{-1}(p)$, and it is a Lie group (see for example \cite[Proposition~1.1.3]{Wang2018}).

Given a subset $S\subset M$ we say that it is \emph{saturated} if for any $p\in S$ we have $L_p\subset S$.

The following list has the goal of providing the reader with specific examples and to illustrate how broad is the notion of a Lie groupoid.

\begin{example}[Lie groups]
A Lie group $\G=G$ is a Lie groupoid  $\G\rightrightarrows \{\ast\}$, with the set of objects a point. In this example, the unit map $\1\colon \{\ast\}\to \G$ is identify with the neutral element of $G$, the multiplication and inversion maps of the Lie groupoid are given by the multiplication and inverse maps of then Lie group. Moreover, there is only one orbit $L_\ast = \{\ast\}$, and the isotropy groups are the Lie group $G$.
\end{example} 

\begin{example}[Manifolds]
A smooth manifold $M$, can be considered as the arrow and object space of a Lie gropoid $\G\rightrightarrows M$ as follows: We set $\G = M$, and the maps $s$, $t$, $\1$, and  $i$ to be the identity map $\Id_M\colon M\to M$. Observe that $\G^{(2)} = \{(p,q)\mid p= s(p)= t(q) = q\} = \Delta\subset M\times M$. By setting $\m = \proj_1\colon \G^{(2)} = \Delta\to M=\G$, the projection onto the first factor, we have that $\G\rightrightarrows M$ is a Lie groupoid. For any $p\in M$ we have $L_p = \{p\}$ and the isotropy groups are trivial, $\G_p = \{\1_p\}$.
\end{example}

\begin{example}[Pair groupoid]\label{Ex: Pair groupoid 1}
Given a smooth manifold $M$, the \emph{pair groupoid of $M$} $\G\rightrightarrows M$ is defined by setting $\G = M\times M$, $s= \proj_1$ the projection onto the $1$st factor, $t=\proj_2$ the projection onto the $2$nd factor, $\1_p = (p,p)$, and $i(p,q) = (q,p)$. It remains to define the multiplication map. Observe that $\G^{(2)} = \{\big((p_3,p_2),(p_2,p_1)\big)\mid p_1,p_2,p_3\in M\}$, and we set $\m\big((p_3,p_2),(p_2,p_1)\big) = (p_3,p_1)$. For any $p\in M$ we have that the orbit is $L_p = \{M\}$, and the isotropy groups are trivial, $\G_p = \{(p,p)\}=\{\1_{p}\}$.
\end{example}

\begin{example}[Submersion groupoid]\label{Ex: Submersion grupoid}
Let $f\colon M\to N$ be a smooth submersion. The \emph{submersion groupoid of $f$}, denoted as $\G_f\rightrightarrows M$, is obtained by setting:
\[
	\G_f = M\times_N M = \{(p,q)\in M\times M\mid f(q) = f(p) \},
\]
and $s,t\colon \G_f\to M$ by $s(p,q) = q$, $t(p,q) = p$, $i(p,q) = (q,p)$, and $\1\colon M\to \G_f$ as $\1_p = (p,p)$. Observe that 
\[
\G_f^{(2)} = \left\{\big((p_3,p_2),(p_2,p_1)\big)\mid (p_3,p_2),(p_2,p_1)\in \G\right\},
\] 
and we set $\m\big((p_3,p_2),(p_2,p_1)\big) = (p_3,p_1)$.

Here we have for any $p\in M$ that $L_p = t(\{(q,p)\in M\times M\mid q\in f^{-1}(p)\}) = f^{-1}(f(p))$, i.e. the orbits of $\G_f\rightrightarrows M$ is given by the fibers of the submersion $f$. Moreover $(\G_f)_p = \{\1_p\}$.
When $N = \{\ast\}$, then the submersion groupoid $\G_f\rightrightarrows M$ is the pair groupoid.
\end{example}

\begin{example}[Restriction groupoid]
Let $\G\rightrightarrows M$ be a Lie groupoid, and let $S\subset M$ be a smooth submanifold $S\subset M$.  Since both $s$ and $t$ are submersions, we have that $s^{-1}(S)$ and $t^{-1}(S)$ are smooth submanifolds of $\G$. When $s^{-1}(S)\cap t^{-1}(S)$ is also a submanifold,  we can define the \emph{restriction groupoid} $\G_S\rightrightarrows S$ as follows: We set $\G_S = s^{-1}(S)\cap t^{-1}(S)$, with the structure maps given by setting  restricting the structure maps of $\G\rightrightarrows M$, to $\G_S$, $\G_S^{(2)}$, and $S$. For example, when  $U\subset M$ is an open subset,  $\G_U$ is an open subset of $\G$, and thus we can define the  restriction groupoid  $\G_U\rightrightarrows U$.

Observe that a subset  $S\subset M$  is a saturated subset if and only if  $s^{-1}(S)=t^{-1}(S)$. Thus we can define the restriction groupoid $\G_S\rightrightarrows S$ for saturated submanifold $S$, and the codimension of $\G_S$ in $\G$ is the same as the codimension of $S$ in $M$. Any leaf $L\subset M$ of $\G\rightrightarrows M$ is a saturated submanifold.
\end{example}

\begin{remark}
When we consider $S$ to be an orbit $L\subset M$ of the Lie groupoid $\G\rightrightarrows M$, then by \cite[Section 3.4]{delHoyo2013}, we get the restriction groupoid $\G_L\rightrightarrows L$.
\end{remark}

\begin{example}[Tangent groupoid of a Lie groupoid]
Given a Lie groupoid $\G\rightrightarrows M$ with structure maps $\{s,t,\1,i,\m\}$, we define the \emph{tangent groupoid} $T\G\rightrightarrows TM$, by setting the structure maps to be $\{s_\ast,t_\ast,\1_\ast,i_\ast,\m_\ast\}$, i.e. the derivatives of the structure maps of $\G\rightrightarrows M$. 
\end{example}

\begin{example}[Groupoid induced by a group action]
Given $H$ a Lie group, $M$ a smooth manifold, and $\mu\colon H\times M\to M$ a smooth  group action of $H$ on $M$, we can encode the action in the \emph{groupoid induced by $\mu$},  $\G\rightrightarrows M$, as follows: We set $\G = H\times M$, and $s,t\colon \G\to M$ as $s(h,p) = p$, $t = (h,p) = \mu(h,p)$,  $i(h,p) = (h^{-1},\mu(h,p))$, and $\1\colon M\to \G$ as $\1_p = (e,p)$ where $e\in H$ is the neutral element of the group. Observe that $\G^{(2)} = \{((h_2,\mu(h_1,p)),(h_1,p))\mid p\in M,\ h_1,h_2\in H\}$. We set the multiplication map to be $\m \Big(\big(h_2,\mu(h_1,p)\big),(h_1,p)\Big) = (h_2h_1,p)$. Then $\G\rightrightarrows M$ is a Lie groupoid, and for $p\in M$ we have that the orbit of the Lie groupoid is the orbit of the group action, that is $L_p = H(p)$, and the isotropy groups of the Lie groupoid are the isotropy groupoids of the group action, i.e. $\G_p = H_p$. 
\end{example}

\subsection{Left groupoid action} As in the case of Lie groups and manifolds, we can define Lie groupoid actions on bundles as follows. 

Consider $\G\rightrightarrows M$ a Lie groupoid, $P$ a smooth manifold and $\alpha\colon P\to M$ a smooth map. A \emph{left action of $\G\rightrightarrows M$ on $P$ along $\alpha$} consists of the space of arrows 
\[
\G\times_M P = \{(g,p)\in \G\times P\mid s(g) = \alpha(p)\}\subset \G\times P
\] 
and a map $\mu\colon \G\times_M P\to P$ such that: 
\begin{enumerate}[(i)]
\item $\alpha(\mu(g,p)) = t(g)$ for all $(g,p)\in \G\times_M P$.
\item $\mu(g,\mu(h,p)) = \mu(\m(g,h),p)$ for all $(g,\mu(h,p))$, $(h,p)\in \G\times_M P$ and $(g,h)\in \G^{(2)}$.
\item $\mu(\1_{x},p) = p$, for all $x\in M$ with $x = \alpha(p)$.
\end{enumerate}
We call $\mu$ the \emph{action morphism} and $\alpha$ is called the \emph{moment map of the action}. 

We observer that any action $\mu$ of $\G\rightrightarrows M$ on $\alpha\colon P\to M$ realizes the arrows of the groupoid $\	G \rightrightarrows M$ as symmetries of the family of fibers of the moment map $\alpha$. That is, for each $g\in \G$ we define the diffeomporphism $\mu_g\colon P \supset \alpha^{-1}(s(g))\to \alpha^{-1}(t(g))\subset P$ as $\mu_g(p ) = \mu(g,p)$.

When $\alpha\colon P\to M$ is a vector bundle we denote by $P_b = \alpha^{-1}(b)$ the fiber over $b\in M$. The vector bundle is $\alpha\colon P\to M$ is called a \emph{left  representation of $\G\rightrightarrows M$} if we have a left action $\mu\colon \G\times_M P\to P$ of $\G\rightrightarrows M$ along $\alpha$, and the action is such that for each $g\in \G$ the map $\mu_g\colon P_{s(g)}\to P_{t(g)}$ is linear. 

Again, for the sake of completeness, we provide a list of examples of Lie groupoid actions

\begin{example}[Action groupoid]
A left action of a Lie groupoid $\G\rightrightarrows M$ on $P$ along $\alpha\colon P\to M$ given by a map $\mu\colon \G\times_M P\to P$ induces a new Lie groupoid. We set the space of arrows to be $\G\times_M P$, the space of objects to be $P$, and define the source and target maps $\overline{s},\overline{t}\colon \G\times_M P\to P$ by setting $\overline{s}(g,p) = p$, and $\overline{t}(g,p) =\mu(g,p)$. The unit map $\overline{\1}\colon P\to \G\times_M P$ is given by $\overline{\1}_p = (\1_{\alpha(p)},p)$, and the inversion map $\overline{i}\colon \G\times_M P\to \G\times_M P$ as $\overline{i}(g,p) = (i(g),\mu(g,p))$. Moreover we have
\begin{linenomath}
\begin{align*}
(\G\times_M P)^{(2)} = \bigg\{\Big(\big(g_2,\mu(g_1,p)\big),(g_1,p)\Big)\mid& (g_2,g_1) \in \G^{(2)}\mbox{ and } (g_1,p)\in \G\times_M P\bigg\}. 
\end{align*}
\end{linenomath}

We set the multiplication map $\overline{\m}\colon(\G\times_M P)^{(2)}\to \G$ as 
\[
\overline{\m}\Big(\big(g_2,\mu(g_1,p)\big),(g_1,p)\Big) = \big(\m(g_2,g_1),p\big).
\] 
With these structure maps we have a Lie groupoid $\G\times_M P\rightrightarrows P$, called the \emph{action Lie groupoid}. Given $p\in P$ the orbit through $p$ is $L_p = \{\mu(g,p)\mid g\in s^{-1}(\alpha(p))\}$, and the isotropy groups are given by $(\G\times_M P)_p = \G_{\alpha(p)}\times \{p\}$. 
\end{example}

\begin{example}[Lie group actions as Lie groupoid actions]
When we consider a Lie group $G=\G\rightrightarrows \{\ast\}$,  and a manifold $M$ with the trivial map $\alpha \colon M\to \{\ast\}$, then a Lie groupoid action of $\G\rightrightarrows \{\ast\}$ on $M$ along $\alpha$ is given by a group action $\mu\colon \G\times M\to M$. Moreover, the action groupoid of the action of $\G\rightrightarrows \{\ast\}$ on $M$ along $\alpha$ is the group action of $G$ on $M$. 
\end{example}

\begin{example}[Canonical Lie groupoid action]\th\label{E: canonical Lie groupoid action}
Given a Lie groupoid $\G\rightrightarrows M$, we define the \emph{canonical Lie groupoid action} of $\G\rightrightarrows M$ on $\Id\colon M\to M$ to be the map $\mu\colon \G\times_M M\to M $ given by 
\[
\mu(g,s(g)) = t(g).
\]
\end{example}

\begin{remark}
Observe that in \th\ref{E: canonical Lie groupoid action} the orbits of $\G\times_M M\rightrightarrows M$ coincide with the orbits of $\G\rightrightarrows M$.
\end{remark}

\begin{example}[Left Lie groupoid action]\th\label{E: Left Lie groupoid action}
Given a Lie groupoid $\G\rightrightarrows M$, we define the \emph{left Lie groupoid action} of $\G\rightrightarrows M$ on $t\colon \G\to M$, to be the map $\mu\colon \G\times_M \G\to M$ defined for $(g,h)\in \G^{(2)}$ as
\[
\mu(g,h) = \m(g,h).
\]
\end{example}

A Lie group action $\mu$ of $\G\rightrightarrows M$ on $\alpha\colon P\to M$ is called \emph{free} if the action groupoid $\G\times_M P\rightrightarrows P$ has trivial isotropy groups. The action is called \emph{proper} if the map $\G\times_M P\to P\times P$ defined by $(g,p)\mapsto (\mu(g,p),p)$ is a proper map. We define the orbit space of the action $P/\G$ as the quotient space given by the partition of $P$ by the orbits of the action groupoid $\G\times_M P\rightrightarrows P$. Moreover henceforth, when we refer to the orbit space $P/\G$, we also consider it a topological space equipped  with the quotient topology induced by the topology of $P$, and the orbit decomposition given by $\G\times_M P\rightrightarrows P$.

\subsection{Singular Riemannian foliations}

\begin{definition}\th\label{D: Singular Riemannian foliation}
A \emph{singular Riemannian foliation} $\fol=\{L_p\mid p\in M\}$ on a Riemannian manifold $(M,g)$ is a partition of $M$ into injectively immersed submanifolds $L_p$, called leaves, such that it is:
\begin{enumerate}[(i)]
\item a smooth singular foliation: There exists a family $\{X_\alpha\}$ of smooth vector fields of $M$ such that for any $p\in M$ with $p\in L_p$ we have $\mathrm{Span}\{X_\alpha(p)\} = T_p L_p$;
\item a transnormal system: For any geodesic $\gamma\colon [0,1]\to M$ with $\gamma'(0)\perp T_{\gamma(0)} L_{\gamma(0)}$, we have that $\gamma'(t)\perp T_{\gamma(t)} L_{\gamma(t)}$.
\end{enumerate}
\end{definition}

We say that a singular Riemannian foliation is \emph{closed} if all the leaves are closed manifolds (i.e. compact without boundary); in this case the leaves are embedded submanifolds. A subset $U\subset M$ is called \emph{saturated} if for any $p\in U$ we have $L_p\subset U$.

Fix a point $p\in M$ and $\varepsilon<\mathrm{injrad}(g)(p)$. We define  the \emph{normal disk tangent disk at $p$ of radius $\varepsilon$} by  
\[
\D^\perp_p (\varepsilon) = \{v\in T_p M\mid \|v\|^2_g\leq \varepsilon\}.
\]
Due to \cite{Molino}, for fixed $v\in \D^\perp_p (\varepsilon)$ by setting $q = \exp_p(v)\in L_{\exp_p(v)}$, then we have that $d(L_p,L_q) = d(L_p,q)=\|v\|_g$. That is, given a sufficiently closed leaf $L_q$ to $L_p$, then $L_q$ is contained in the boundary of a distance tubular neighborhood of $L_p$. By \cite{Molino}, the inverse image under $\exp_p$ of the connected components of the intersections of leaves with $\exp_p(\D^\perp(\varepsilon))$ gives a smooth foliation $\fol_p(\varepsilon)$ on $\D^\perp_p(\varepsilon)$. Moreover, $\fol_p(\varepsilon)$ is a singular Riemannian foliation with respect to the (restricted) Euclidean inner product $g_p$ on $T_p M$. Given $\varepsilon'\leq \varepsilon$, the foliated disks $(\D^\perp_p(\varepsilon'),\fol_p(\varepsilon'))$ and $(\D^\perp_p,\fol_p(\varepsilon))$ are foliated diffeomorphic via the homothety $h_{\varepsilon/\varepsilon'}\colon T_p M\to T_p M$ defined as $h_{\varepsilon/\varepsilon'}(w) := (\varepsilon/\varepsilon') w$. Thus the foliation $\fol_p(\varepsilon)$ does not depend on the scale $\varepsilon$. Thus we may assume without loss of generality that $\fol_p$ is defined on the normal unit disk $\D^\perp_p:=\D^\perp_p(1)$. We call this foliation \emph{the infinitesimal foliation at $p$}.

We also have that the leaves of $\fol_p(\varepsilon)$ are contained in spheres in $\D^\perp_p(\varepsilon)$. That is, given $v\in \D^\perp_p(\varepsilon)$, then the subset $\Sp^\perp_p(\|v\|_{g}) := \{w\in \D^\perp_p(\varepsilon)\mid \|w\|^2_{g}= \|v\|^2_{g}\}$ is saturated by $\fol_p(\varepsilon)$. Since $\fol_p^\perp(\varepsilon)$ is independent of $\varepsilon$, without loss of generalization, we may assume that $\fol_p$ restricted to the boundary sphere $\Sp^\perp_p=\partial \D^\perp$ induces a singular folation,  also called the infintesimal foliation, and which we also denote by $\fol_p$. Moreover by \cite{Molino}, the foliation $(\Sp^\perp_p,\fol_p)$ is a singular Riemannian foliation with respect to the round metric on $\Sp^\perp_p$ induced from the Euclidean metric on $\D^\perp$. Observe that via homotheties, we can recover $(\D^\perp_p,\fol_p)$ from $(\Sp^\perp_p,\fol_p)$.

Consider by $\Omega_p^\ast$ the set of all piece-wise smooth loops based at $p$. Then by \cite{MendesRadeschi2015}, for each $\gamma\in \Omega_p^\ast$ there exists a foliated isometry $\Gamma_\gamma\colon (\D^\perp_p,\fol_p)\to (\D^\perp_p,\fol_p)$. The group $\mathrm{Hol}(L,p):= \{\Gamma_\gamma\mid\gamma\in \Omega_p^\ast\}$ is a Lie group \cite{Radeschi-notes}: the connected component of $\mathrm{Hol}(L,p)$ containing the identity the map, denoted by $\mathrm{Hol}(L,p)_0$, is a path connected subgroup of the Lie group $\mathrm{O}(\Sp^\perp_p)$ and thus a Lie group by \cite{Yamabe1950}. The map $\overline{\mathrm{Hol}}\colon \pi_1(L,p)\to \mathrm{Hol}(L,p)/\mathrm{Hol}(L,p)_0$ is a  well defined surjective group morphism, \cite[Corollary A.5]{Corro}. In particular $|\pi_0(\mathrm{Hol}(L,p))|= |\mathrm{Hol}(L,p)/\mathrm{Hol}(L,p)_0|\leq |\pi_1(L,p)|$. But since $L$ is a smooth manifold, $\pi_1(L,p)$ is countable set \cite[Prop. 1.16]{Lee}. This implies that $\mathrm{Hol}(L,p)$ is a group whose number of connected components are countable, and the connected components are Lie groups. Thus, by definition $\mathrm{Hol}(L,p)$ is a smooth manifold. Then it follows that $\mathrm{Hol}(L,p)$ is a Lie group.

If $q\in L$ is in the same connected component as $p$ we have that
\[
\mathrm{Hol}(L,q)= G^{-1}\mathrm{Hol}(L,p)G
\]
for some foliated isometry $G\colon (\Sp_q^\perp,\fol_q)\to (\Sp^\perp_p,\fol_p)$. Thus for a singular Riemannian foliation $(M,\mathcal{F})$ with connected leaves, the holonomy of a leaf $L$ depends on the base point $p$ up to conjugation. 

\begin{thm}[Slice Theorem, \cite{MendesRadeschi2015}]\th\label{T: Slice Theorem}
Let $(M,\fol)$ be a closed singular Riemannian foliation and $L\in \fol$. Then, there exists a sufficiently small $r>0$ such that a distance tubular neighborhood $B_r(L)$ is saturated and it is foliated diffeomorphic to 
\[
	(P\times_{\mathrm{Hol}(L,p)} \D^\perp_p,\D^\perp_p\times_{\mathrm{Hol}(L,p)} \fol_p),
\]
where $P$ is a $K$-principal bundle over $L$. For $[\xi,v]\in P\times_{\mathrm{Hol}(L,p)} \D^\perp_p$ we have $[\zeta,w]\in L_{[\xi,v]}$ if and only if there exists $h\in K$ such that $h(v)=w$.
\end{thm}

Moreover, we can identify $B_r(L)$ with the normal bundle $\nu(L)\to L$ and there exists an Ehresmann connection $H$, which induces a connection $\nabla^H$  on $\nu(L)$ such that $\mathrm{Hol}(\nabla^H,p)= \mathrm{Hol}(L,p)$ (see \cite{AlexandrinoInagakiStruchiner2018} and \cite{Alexandrino2010}). 

We denote by $\mathrm{O}(\Sp^\perp,\fol_p): = \{h\in \mathrm{Iso}(\Sp^\perp_p)\mid h(L_v)\subset L_{h(v)}\}$ the isometries of $(\Sp^\perp_p,\fol_p)$ that map leaves to leaves, and we set $\mathrm{O}(\fol_p):= \{h\in \mathrm{O}(\Sp^\perp,\fol_p)\mid h(v)\in L_v\}$. We observe that if $\dim(\Sp^\perp_p) = k_p$, then $\mathrm{Iso}(\Sp^\perp_p)=\mathrm{O}(k_p)$.

Let $K_p\subset \mathrm{O}(\fol_p)$ be the maximal connected Lie subgroup. Let $\mathrm{Hol}^H\rightrightarrows \nu(L)$ denote the Lie groupoid generated by all parallel translations with respect to $\nabla^H$ along piece-wise smooth curves $\gamma\colon [0,1]\to L$. We define \emph{the linearized foliation $\mathcal{F}^\ell_L$ of $\mathrm{fol}$ around $L$} as the foliation
\[
\mathcal{F}^\ell_L := \{\{P_\alpha(gv)\mid P_\alpha\mbox{ parallel transport along }\alpha\colon I\to L,\: g\in K_p,\: v\in \D^\perp_p\}.
\]
Via the normal exponential map the foliation $\mathcal{F}^\ell_L$ on $\nu(L)$ induces a foliation, also denoted $\mathcal{F}^\ell_L$, on $B_r(L)$. Then we have that for $L^\ell\in \fol^\ell_L$ that $L^\ell\subset L_v$ for $v\in L^\ell$, i.e. $\fol^\ell_L\subset \fol|_{B_r(L)}$. A closed singular Riemannian foliation $\fol$ on a compact manifold $M$ is called an \emph{orbit-like} foliation if for any $L\in \fol$ we have $\fol^\ell_L= \fol$ in some small tubular neighborhood of $L$. Observe that this is equivalent to the infinitesimal foliation $\fol_p$ being homogeneous at any point (see \cite{AlexandrinoRadeschi2016}). 

Then we have the following theorem: 

\begin{thm}[Theorem 1.5 in \cite{AlexandrinoInagakiStruchiner2018}]\th\label{T: Local transversal homogeneous subfoliation}
Let $(M,\fol,g)$ be a closed singular Riemannian foliation on a complete Riemannian manifold. Then for $L\in \fol$ there exists a saturated small tubular neighborhood $B_r(L)$, a Lie groupoid $\mathcal{G}\rightrightarrows L$ and a left represenation of this Lie groupoid on $\nu(L)\to L$ such that by identifying a small saturated tubular neighborhood of the $0$-section on $\nu(L)$ with $B_r(L)$, the orbits of the groupoid $\mathcal{G}\times_L \nu(L)\rightrightarrows \nu(L)$ on this tubular neighborhood can be identified with the leaves of $\fol^\ell_L$ on $B_r(L)$.
\end{thm}

\subsection{General smooth singular foliations}

A \emph{smooth singular foliation} $(M,\fol)$ is a partition of $M$ into smooth submanifold satisfying \th\ref{D: Singular Riemannian foliation} (i). 

For the case when the foliation has closed leaves all of the same dimension, i.e. it is a closed regular foliation, there exists a Lie groupoid, the so called \emph{holonomy groupoid}, defining the foliation (see \cite{Ehresmann1965} \cite{Winkelnkemper1983}). 

In the case of a general smooth singular Riemannian foliations, there exist always a non necessarily Hausdorff topological holonomy groupoid inducing the foliation by \cite[Theorem 0.1]{AndroulidakisSkandalis2009}. In some particular cases this topological holonomy groupoid is actually a Lie groupoid, see \cite{Debord2001}. But as pointed out in \cite[p. 496, Examples 3.]{Debord2001} there are singular Riemannian foliations, such as the foliation of $\R^3$ by concentric $2$-spheres and the origin, for which such a global holonomy Lie groupoid may not exist. Nonetheless, by \cite{AlexandrinoInagakiStruchiner2018} this example is given by a Lie groupoid action. 

\subsection{Normal representation} We begin by pointing out that given a smooth group action $\mu\colon H\times M\to M$ of a Lie group $H$ on a smooth manifold $M$, we have a lift of the action to the tangent bundle of $M$, $\mu_\ast\colon H\times TM\to TM$. This action is given as follows: For $h\in H$ fixed, we can consider the diffeomorphism $\mu_h\colon M\to M$ defined as $\mu_h(p) = \mu(h,p)$. We then set $\mu_\ast(h,X) = (\mu_h)_\ast (X)$, for $h\in H$ and $X\in TM$. 

In contrast, for a Lie groupoid action $\G\rightrightarrows M$ on $\alpha\colon P\to M$ such a natural lift does not exists. But we have the so called \emph{normal representation} of the action groupoid $\G\times_M P \rightrightarrows P$.

Given two Lie groupoids $\G\rightrightarrows M$ and $\G'\rightrightarrows M'$, we define a \emph{morphism of Lie groupoids} $\phi\colon (\G\rightrightarrows M)\to (\G'\rightrightarrows M')$ to be two smooth maps $\phi^{\ar}\colon \G\to \G'$ and $\phi^{\ob}\colon M\to M'$ such that for the structure maps $\{s,t,\1,i,\m\}$ of $\G\rightrightarrows M$ and $\{s',t',\1',i',\m'\}$ of $\G'\rightrightarrows M'$ we have $s'\circ \phi^{\ar} = \phi^{\ob}\circ s$, $t'\circ \phi^{\ar} = \phi^{\ob}\circ t$, $i'\circ \phi^{\ar} = \phi^{\ar}\circ i$, $\1'\circ\phi^{\ob} = \phi^{\ar}\circ \1$, and for $(g,h)\in \G^{(2)}$ we have $\mu'(\phi^{\ar}(g),\phi^{\ar}(h)) = \phi^{\ar}(\m(g,h))$.

An \emph{VB-groupoid} consists of two groupoids $\Gamma\rightrightarrows E$ and $\G\rightrightarrows M$ and a groupoid map $\phi\colon (\Gamma\rightrightarrows E) \to (\G\rightrightarrows M)$, such that the maps $\phi^{\ar}\colon \Gamma\to \G$ and $\phi^{\ob}\colon E\to M$ are vector bundles, and the structure maps of $\Gamma\rightrightarrows E$ are vector bundle maps. We define the \emph{core} of the VB-groupoid to be $C = \kernel(s\colon \Gamma\to E)|_M$, where we identify $M$ with the $0$-section of $\phi^{\ob}\colon E\to M$. By the following proposition, VB-groupoids are equivalent to linear representations of Lie groupoids.

\begin{proposition}[Proposition~3.5.5 in \cite{delHoyo2013}]\th\label{P: correspondence between VB-groupoids and representations}
There is a $1$--$1$ correspondence between linear representations of a Lie groupoid $G\rightrightarrows M$ and VB-groupoids $(\Gamma\rightrightarrows E) \to (\G\rightrightarrows M)$  with trivial core.
\end{proposition}

\begin{example}
A Lie groupoid $\G\rightrightarrows M$ together with its tangent groupoid $T\G\rightrightarrows TM$ and  the groupoid morphism $\pi\colon (T\G\rightrightarrows TM)\to \G\to M$ induced by the projections $\pi_{\G}\colon T\G\to \G$ and $\pi_M\colon TM\to M$, is a VB-groupoid. 
\end{example}

We now consider $\G\rightrightarrows M$ a Lie groupoid, and fix an orbit $L\subset M$. For the restriction groupoid $\G_{L}\rightrightarrows L$ the embedding $L\subset M$ induces an embedding $\G_L\subset \G$. With this embeddings we define the normal bundles $N(\G_L)\to \G_L$ and $N(L)\to L$ by setting $N(\G_L) = T\G|_{\G_L}/T\G_L$ and $N(L) = TM|_L/TL$. With this we obtain the following sequence of VB-groupoids
\[
(T\G_L)\rightrightarrows TL)\to (T\G|_{\G_L}\rightrightarrows TM|_{L})\to (N(\G_L)\rightrightarrows N(L)).
\]
Observe that as with the tangent groupoid of a Lie groupoid, we have a VB-groupoid 
\[
\pi\colon (N(\G_L)\rightrightarrows N(L))\to (\G_L\rightrightarrows L).
\] 
Moreover, since $\G_L\subset \G$ and $L\subset M$ have the same codimension, this implies that the bundles $N(\G_L)\to \G_L	$ and $N(L)\to L$ have the same rank.  From this it follows that the core of the VB-groupoid 	$\pi\colon (N(G_L)\rightrightarrows N(L))\to (\G_L\rightrightarrows L)$	has trivial core. From \th\ref{P: correspondence between VB-groupoids and representations} we have  a linear representation	$N(\mu)$ of $\G_L\rightrightarrows L$ on \linebreak$\pi_{N(L)}\colon N(L)\to L$, called the \emph{normal representation of $\G_L\rightrightarrows L$}, or the \emph{normal representation of $\G\rightrightarrows M$ around $L$}.	

We can also express the normal representation as follows (see \cite[p.~178]{delHoyo2013}). Consider $q\in L_p$, and $g\in s^{-1}(q)$. For $v\in N_q(L_p)$ we consider a curve $\gamma\colon I \to M$ with $\gamma(0) = q$, whose velocity $\gamma'(0)$ represents the class  $v$, i.e. $[\gamma'(0)] = v$ in $T_q M|_{L_p}/T_q L_p$. Let $\widetilde{\gamma}\colon I \to \G$ be a curve such that $\widetilde{\gamma}(0) = g$ and $s\circ \widetilde{\gamma} = \gamma$. Such a curve always exists, since $s\colon \G\to M$ is a submersion. Then we set $N(\mu)(g,v)$ to be $(t\circ \widetilde{\gamma})'(0)$, that is:
\[
N(\mu)(g,v) = N(\mu)\big(g,[D_g s(\gamma'(0))] \big) = [D_g t(\gamma'(0))].
\]

\subsection{Riemannian Lie groupoids and invariant actions}\label{S: Riemannian Lie groupoids} In this section we present the notions of Riemannian metrics associated to Lie groupoids presented in \cite{delHoyoFernandes2018}.

Given a Lie groupoid action $\mu$ of $\G\rightrightarrows M$ on $\alpha\colon P\to M$, we can consider $L\subset P$ an orbit of the action groupoid $\G\times_M P\rightrightarrows P$. Given a Riemannian metric $\eta^P$ on $P$, we can identify the total spaces of the  normal bundle $N(L)\to L$ with $\nu(L):= \{v\in TP|_{L}\mid\forall x\in TL,\: \eta^{P}(v,x) =0\}\subset TP$. 

\begin{definition}\th\label{D: mu-transversly invariant}
The Riemannian metric $\eta^P$ is called \emph{transversely $\mu$-invariant} if for the  normal representation $N(\mu)$ of the action groupoid $(\G\times_M P)_L \rightrightarrows L$ on $\nu(L)\to L$  is by isometries, i.e. for $(g,p)\in (G\times_M P)_L$ the map $N(\mu)_{(g,p)}\colon \nu_{p}(L) \to \nu_{\mu(g,p)}(L)$ is an isometry.
\end{definition}

A \emph{$0$-metric} on a Lie groupoid $\G\rightrightarrows M$ is a Riemannian metric $\eta^{(0)}$ on $M$, such that it is transversely $\mu$-invariant for $\mu$ the canonical action of the Lie groupoid (see \th\ref{E: canonical Lie groupoid action}). 

These metrics relate to Riemannian submersions by the following proposition.

\begin{proposition}
Given $f\colon M\to N$ a submersion, a $0$-metric on the submersion groupoid $M\times_N M\rightrightarrows M$ is a Riemannian metric $\eta^{(0)}$ making $f$ into a Riemannian submersion (see \cite[Example 3.4]{delHoyoFernandes2018}.
\end{proposition}

Since $0$-metrics are only on the object space of a Lie groupoid, we can also consider Riemannian metrics on the space of arrows. A \emph{$1$-metric} on a Lie groupoid $\G\rightrightarrows M$ is a Riemannian metric $\eta^{(1)}$ on $\G$ which is transversely $\mu$-invariant for the left Lie groupoid action of $\G\rightrightarrows M$ on $t\colon \G\to M$ (see \th\ref{E: Left Lie groupoid action}), and such that the inversion map $i\colon \G\to \G$ is an isometry. 

\begin{remark}
An equivalent definition of a $1$-metric, is that $\eta^{(1)}$ is a Riemannian metric on $\G$ such that the map $s\colon \G\to M$ is a Riemannian submersion (see \cite[Example 2.9]{delHoyo2013}).
\end{remark}

From the following proposition it follows that $1$-metrics induce $0$-metrics.

\begin{proposition}[Proposition 3.8 in \cite{delHoyoFernandes2018}]\th\label{P: 1-metrics induce 0-metrics}
Given a Lie groupoid $\G\rightrightarrows M$, a $1$-metric $\eta^{(1)}$, there exists a Riemannian metric $\eta^{(0)}$ on $M$ such that:
\begin{enumerate}[(i)]
\item The source and target maps $s,t\colon \G\to M$ are Riemannian submersions, and the submanifold $\1(M)\subset \G$ is a totally geodesic submanifold.
\item The foliation $\fol_{\G}$ on $M$ induced by the orbits of the groupoid is a singular Riemannian foliation with respect to $\eta^{(0)}$.
\item $\eta^{(0)}$ is  a $0$-metric.
\end{enumerate}
\end{proposition}

A Lie groupoid $\G\rightrightarrows M$ is called a \emph{foliation groupoid}, if for every $p\in M$, the  isotropy group $\G_p$ is discrete.

Conversely every $0$-metric induces a $1$-metric:
\begin{proposition}\th\label{P: extension of 0-metrics}
If $G\rightrightarrows M$ is a foliation groupoid and $\eta^{(0)}$ is a $0$-metric, then there exists a $1$-metric $\eta^{(1)}$ on $\G$, such that $\eta^{(0)}$ is  the metric induced in \th\ref{P: 1-metrics induce 0-metrics} by $\eta^{(1)}$.
\end{proposition}

\begin{remark}
When we consider $\G= G\to \{\ast\}$ a Lie group, then any Riemannian metric on $G$ induces a $1$-metric on $\G$. Thus, we observe that $1$-metrics do not account for the Lie groupoid multiplication. To consider the multiplication, it is necessary to consider the space of composable arrows $\G^{(2)}$ of the Lie groupoid $\G\rightrightarrows M$.
\end{remark}

As done in \cite[Section~3.3]{delHoyoFernandes2018}, we can identify a pair of composable arrows $(g,h)$, and their multiplication $\m(g,h)$ with a diagram as follows:
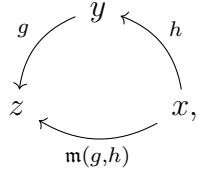
\begin{center}
\begin{tikzcd}[column sep = small]
	&y\arrow[bend right,swap]{dl}{g}&\\
	z & &x\arrow[bend left]{ll}{\m(g,h)}\arrow[bend right,swap]{ul}{h},
\end{tikzcd}
\end{center}
where $x = s(h) = s(\m(g,h))$, $y= t(h)=s(g)$ and $z= t(g) = t(\m(g,h))$. With this identification, we can define an action of the symmetric group $S_3$ on $\G^{(2)}$, by identifying the source $s(h)$ with $1$, the target $t(h)$ with $2$, and the target $t(g)$ with $(3)$. So for example, the element $(1,3)\cdot (g,h)$ corresponds to the diagram
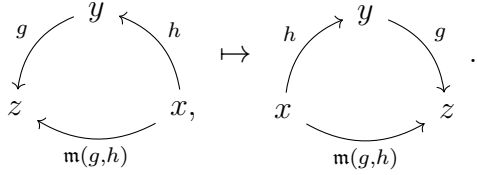
\[
\begin{tikzcd}[column sep = small]
	&y\arrow[bend right,swap]{dl}{g}&\\
	z & &x\arrow[bend left]{ll}{\m(g,h)}\arrow[bend right,swap]{ul}{h},
\end{tikzcd}\mapsto \begin{tikzcd}[column sep = small]
	&y\arrow[bend left]{dr}{g}&\\
	x\arrow[bend left]{ur}{h} \arrow[bend right,swap]{rr}{\m(g,h)} & &z
\end{tikzcd}.
\]
That is $(1,3)\cdot (g,h) = (i(h),i(g))$ (see \cite[Remark 3.13]{delHoyoFernandes2018} for an alternative definition). 

We have also three commuting groupoid actions $\mu_1$, $\mu_2$ and $\mu_3$ of $(\G\rightrightarrows M)$ on $\alpha_i\colon \G^{(2)}\to M$, $i\in \{1,2,3\}$, defined as follows:
\begin{enumerate}[(i)]
\item
\[
\mu_1\big(k,(g,h)\big) = \big(\m(k,g),h\big)\quad \mbox{and}\quad \alpha_1(g,h) = t(g),
\]
\item
\[
\mu_2\big(k,(g,h)\big) = \big(\m(g,i(k)),\m(k,h)\big)\quad \mbox{and}\quad \alpha_2(g,h) = t(h),
\]
\item
\[
\mu_3\big(k,(g,h)\big) = \big(\m(k,g),h\big)\quad \mbox{and}\quad \alpha_2(g,h) = s(h),
\]
\end{enumerate}
\begin{remark}
Observe that the action of $S_3$ on $\G^{(2)}$ interchanges the actions $\mu_i$ of $\G\rightrightarrows M$ on $\alpha_i\colon\G^{(2)}\to M$. Moreover, the orbits of the actions $\mu_1$, $\mu_2$, and $\mu_3$, i.e. the orbits of the action groupoids $\G\times_M \G^{(2)}\rightrightarrows \G^{(2)}$, are the fibers of the maps $\proj_2,\m,\proj_1\colon \G\times \G\supset \G^{(2)}\to \G$ respectively.
\end{remark}

A Riemannian metric $\eta^{(2)}$ on $\G^{(2)}$ is called a \emph{$2$-metric} on $\G\rightrightarrows M$ if it is transversely invariant for the $\mu_1$-action of  $\G\rightrightarrows M$ on $\alpha_1\colon \G^{(2)}\to M$ and the action of $S_3$ is by isometries. The pair $(\G\rightrightarrows M, \eta^{(2)})$ is called a \emph{Riemannian groupoid}. In the literature, the term ``Riemannian groupoid'' also refers to a $1$-metric on the Lie groupoid $\G\rightrightarrows M$, cf.\ \cite{GallegoGualdraniHectorReventos1989,Glickenstein2008}.

\begin{remark}
From the fact that $S_3$ permutes the actions $\mu_1$, $\mu_2$ and $\mu_3$, if follows that the definition of a $2$-metric $\eta^{(2)}$ on $\G^{(2)}$ is equivalent to the following definitions:
\begin{itemize}[itemsep = 0.5em]
\item $\eta^{(2)}$ is transvesrsely $\mu_i$-invariant for any of the actions $\mu_i$ of $\G\rightrightarrows M$ on $\alpha_i\colon \G^{(2)}\to M$, and the group $S_3$ acts by isometries.
\item $\eta^{(2)}$ makes any of the maps $\proj_2,\m,\proj_1\colon \G^{(2)}\to \G$ a Riemannian submersion, and the group $S_3$ acts by isometries.
\end{itemize}
\end{remark}

As before, a $2$-metric on $\G^{(2)}$ induces a $1$-metric on $\G$, and thus by \th\ref{P: 1-metrics induce 0-metrics} a $0$-metric on $\G$.

\begin{proposition}[Proposition 3.16 in \cite{delHoyoFernandes2018}]\th\label{P: 2-metrics induce 1-metrics}
Let $\G \rightrightarrows M$ be a Lie groupoid. A $2$-metric $\eta^{(2)}$ on $\G^{(2)}$ induces a $1$-metric $\eta^{(1)}$ on $\G$.
\end{proposition}

That is, given a $2$-metric $\eta^{(2)}$ on a Lie groupoid $\G\rightrightarrows M$, there are induced metrics $\eta^{(1)}$ on $\G$ and $\eta^{(0)}$ on $M$, such that the following maps are Riemannian submersions
\begin{center}
\begin{tikzcd}
\G^{(2)} \arrow[r,"\m" description]
\arrow[r, shift left=2,"\proj_2"]
\arrow[r, shift right=2,swap,"\proj_1"] & \G \arrow[r, shift left, "s"]
\arrow[r, shift right,swap,"t"] & M.
\end{tikzcd}
\end{center}

\begin{remark}\th\label{R: 2-metrics on Lie groups}
Given a Lie group $\G = G\rightrightarrows \{\ast\}$, a bi-invariant metric $\eta$ on $\G$ is a $1$-metric on $G$, coming from the $2$-metric $\eta\times \eta$ on $\G^{(2)} = G\times G$. Nonetheless there are $1$-metrics on $G\rightrightarrows \{\ast\}$ coming from $2$-metrics which are not bi-invariant metrics (see \cite[Example 4.4]{delHoyoFernandes2018}.
\end{remark}

\subsection{Hausdorff proper Lie groupoids and existence of 2-metr\-ics} As with  $1$- and $0$-metrics, there are sufficient conditions to guarantee the existence of a $2$-metric given a $1$-metric. To state this condition we consider proper Lie groupoids. A Lie groupoid $\G\rightrightarrows M$ is \emph{proper}, if the map $\rho\colon \G\to M\times M$ given by $\rho(g) = (t(g),s(g))$ is a proper map. 

\begin{thm}[Theorem 1 in \cite{delHoyoFernandes2018}]\th\label{P: existence of 2-metrics}
Any proper Lie groupoid $\G\rightrightarrows M$ admits a $2$-metric. Moreover, given a $1$-metric $\eta^{(1)}$ on $\G\rightrightarrows M$, there exists a $2$-metric $\eta^{(2)}$ on $\G\rightrightarrows M$ such that $\eta^{(1)}$ agrees with the induced Riemannian metric given by \th\ref{P: 2-metrics induce 1-metrics}.
\end{thm}

Moreover, given a proper Lie groupoid action $\mu$ of $\G\rightrightarrows M$ on $\alpha\colon P\to M$, and a Riemannian  metric $\eta$ on $P$, we can always give a new Riemannian metric $\tilde{\eta}$ on $P$ which is $\mu$-transversely invariant.

\begin{thm}[Proposition~4.11 in \cite{delHoyoFernandes2018}]\th\label{T: existence of G->M invariant metrics on P->M}
Given any Riemannian  metric $\eta$ on $P$, and a Lie groupoid action $\mu$ of a proper $\G\rightrightarrows M$ Lie groupoid on $\alpha\colon P\to M$, there exists Riemannian metric $\tilde{\eta}$ on $P$ which is $\mu$-transversely invariant. Moreover, in the case when $\eta$ is already $\mu$-invariant, then $\eta$ and $\tilde{\eta}$ agree on the directions normal to the orbits of the action groupoid $\G\times_M P\rightrightarrows P$.
\end{thm}

\begin{remark}
We consider a Lie group $G$ as the Lie groupoid $G\rightrightarrows \{\ast\}$,  $(P,\eta)$ any Riemannian manifold, and a proper Lie group action $\mu\colon G\times P\to P$.  Then we have a proper Lie groupoid action of $G\rightrightarrows \{\ast\}$ along $P\to\{\ast\}$ induced by $\mu$. The metric $\tilde{\eta}$ obtained from \th\ref{T: existence of G->M invariant metrics on P->M} is a $G$-invariant Riemannian metric on $P$. 
\end{remark}

\begin{example}[Pair groupoid]\th\label{Ex: Pair groupoid as a Riemannian groupoid}
Consider $M$ a compact manifold, and let $\G = M\times M\rightrightarrows M$ be the pair groupoid. Fix $g$ a Riemannian metric on $M$. We know observe that the pair groupoid is the submersion groupoid for the submersion $f\colon M\to \{\ast\}$. Moreover, the product metric $g\oplus g$ on $\G = M\times M$ is a $1$-metric. Since $M$ is compact, the pair groupoid $\G= M\times M\rightrightarrows M$ is a proper Hausdorff Lie groupoid. Thus by \cite[Theorem~4.13]{delHoyoFernandes2018} there exists a $2$-metric $\eta^{(2)}$ on $\G^{(2)}$ making the multiplication map $\m\colon (\G^{(2)},\eta^{(2)})\to (\G,g\oplus g)$, as well as the source and target maps $(\G^{(2)},\eta^{(2)})\rightrightarrows\G$ Riemannian submersions. From this example we see that the sectional curvature of a $1$-metric induced by a $2$-metric can be arbitrary.
\end{example}

\section{Cheeger-like deformation for Lie groupoid actions}\label{S: Cheeger like deformation for Lie groupoids actions}

In this section we give the construction on how to deform an action Lie groupoid  onto the orbit space in an analogous fashion to the Cheeger deformation procedure for Lie group actions.

Consider $\mu$ a given action of the proper Lie groupoid $\G\rightrightarrows M$ on $\alpha\colon P\to M$, and the action groupoid $\G\times_M P\rightrightarrows P$. Consider $\eta^P$ a $\mu$-transversely invariant Riemannian metric on $P$ (see \th\ref{T: existence of G->M invariant metrics on P->M}), and $\eta^{(1)}$ a $1$-metric on $\G\rightrightarrows M$ induced by a $2$-metric. Set 
\begin{equation}\label{EQ: Cheeger deformation metric}
\widehat{\eta}_{\varepsilon} = \frac{1}{\varepsilon}\eta^{(1)} + \eta^P
\end{equation}
on $\G\times P$, and restrict it to $\G\times_M P$.

\begin{lemma}
Consider $\mu$ and  action of the proper Lie groupoid $\G\rightrightarrows M$ on $\alpha\colon P\to M$, and the action groupoid $\G\times_M P \rightrightarrows P$. Denote by $\overline{s}\colon \G\times_M P\to P$, and $s\colon \G\to M$ the source maps of $\G\times_M \rightrightarrows P$ and $\G\rightrightarrows P$ respectively. Then $\overline{s}^{-1}(p) = s^{-1}(\alpha(p))\times\{p\}$.
\end{lemma}

\begin{proof}
Recall  that $\overline{s}(g,q) = q$. Thus:
\begin{linenomath}
\begin{align*}
\overline{s}^{-1}(p) = \{(g,p)\mid s(g) = \alpha(p) \} =\{(g,p)\mid g\in s^{-1}(\alpha(p))\}= s^{-1}(\alpha(p))\times \{p\}.
\end{align*}
\end{linenomath}
\end{proof}

From the previous theorem we get the following nice description of the kernel of the submersion $\overline{s}\colon \G\times_M P\to P$.  

\begin{cor}
Consider the action groupoid $\G\times_M P\rightrightarrows P$ of an action of the proper Lie groupoid $\G\rightrightarrows M$ on $\alpha\colon P\to M$. Then $\kernel D_{(g,p)}\overline{s} = \kernel D_g s\times \{0\}$.
\end{cor}

We now show that the source map $\overline{s}\colon \G\times_M P\to P$ of the action groupoid is a Riemannian submersion with respect to the metric $\widehat{\eta}_\varepsilon$.

\begin{thm}\th\label{T: source map for the action groupoid is a Riemannian submersion}
Consider $\G\times_M P\rightrightarrows P$, $\mu$ an Lie groupoid action of the proper Lie groupoid $\G\rightrightarrows M$ on $\alpha\colon P\to M$. Equip the metric $\widehat{\eta}_\varepsilon$ on $\G\times_M P$. Then $\overline{s}\colon (\G\times_M P,\widehat{\eta}_\varepsilon)\to P$ is a Riemannian submersion.
\end{thm}

\begin{proof}
Consider $(X,0)\in \kernel D\overline{s}$, and set $(\kernel D\overline{s})^\perp$ the orthogonal complement to $\kernel D\overline{s}$ with respect to $\eta_\varepsilon$. Consider $(U_1,A_1),(U_2,A_2)\in (\kernel D\overline{s})^\perp$. This implies that for $(Y,0)\in \kernel D\overline{s} = \kernel Ds\times \{0\}$, and $i=1,2$ we have
\[
0 = \widehat{\eta}_\varepsilon \big((Y,0),(U_i,A_i)\big) = \frac{1}{\varepsilon}\eta^{(1)}(Y,U_i)+\eta^{(P)}(0,A_i) = \frac{1}{\varepsilon}\eta^{(1)}(Y,U_i).
\]
From this it follows that $U_1,\: U_2\in (\kernel Ds)^\perp$, where $(\kernel Ds)^\perp$ is the orthogonal complement of $\kernel Ds$ with respect to $\eta^{(1)}$. Then, since $s\colon (\G,\eta^{(1)})\to M$ is a Riemannian submersion we obtain
\begin{linenomath}
\begin{align*}
\Lie_{(X,0)}\left(\frac{1}{\varepsilon}\eta^{(1)}+\eta^P\right )\big((U_1,A_1),(U_2,A_2)\big) =& \frac{1}{\varepsilon}\Lie_{X}(\eta^{(1)})(U_1,U_2)+\Lie_0(\eta^P)(A_1,A_2)\\
=&\frac{1}{\varepsilon}\Lie_{X}(\eta^{(1)})(U_1,U_2)\\
=&0.
\end{align*}
\end{linenomath}
Thus we conclude by \cite[Theorem~1.2.1]{GromollWalschap}, that $\overline{s}$ is a Riemannian submersion with respect to $\widehat{\eta}_\varepsilon$. 
\end{proof}

We know give a description of the fibers of the target map of an action groupoid.

\begin{thm}
Consider $\mu$ a Lie groupoid action of the proper Lie groupoid $\G\rightrightarrows M$ on $\alpha\colon P\to M$, and $\G\times_M P\rightrightarrows P$ the action groupoid. Then for $p\in P$
\[
\overline{t}^{-1}(p) = \{(g,\mu(i(g),p))\mid g\in t^{-1}(\alpha(p))\} = t^{-1}(\alpha(p))\times L_p,
\]
where $L_p\subset P$ is the orbit of the action groupoid $\G\times_M P\rightrightarrows P$ that contains $p$.
\end{thm}

\begin{proof}
Recall that for $(g,p)\in\G\times_M P$ we have $\overline{t}(g,p) = \mu(g,p)$. Observe that \linebreak$\alpha(\mu(i(g),p)) = t(i(g)) = s(g)$, and thus $(g,\mu(i(g),p))\in \G\times_M P$. 

Consider $g\in t^{-1}(\alpha(p))$, then 
\begin{linenomath}
\begin{align*}
\overline{t}\big(g,\mu(i(g),p)\big) = \mu\big(g, \mu(i(g),p)\big) = \mu\big(\m(g,i(g)),p\big) = \mu(\1_{t(g)},p)= \mu(\1_{\alpha(p)},p) = p.
\end{align*}
\end{linenomath}
Thus $\{(g,\mu(i(g),p))\mid g\in t^{-1}(\alpha(p))\}\subset \overline{t}^{-1}(p)$.

Now consider $(g,q)\in \overline{t}^{-1}(p)$. That is $s(g) = \alpha(q)$, and $\mu(g,q) = p$. So we have $\mu(i(g),p) = \mu(i(g),\mu(g,q))$, but observe that  
\[
\mu(i(g),\mu(g,q))=\mu\big(\mu(i(g),g),p\big)= \mu(\1_{s(g)},q) = q.
\] 
This implies that $(g,q) = (g,\mu(i(g),p))$, and thus 
\[
\overline{t}^{-1}(p) = \{(g,\mu(i(g),p))\mid g\in t^{-1}(\alpha(p))\}.
\]

Last we observe that the set $\{\mu(i(g),p)\mid g\in t^{-1}(\alpha(p))\} = \Or_\G(p) = \overline{t}(\overline{s}^{-1}(p))$:
\[
\overline{t}(\overline{s}^{-1}(p)) = \overline{t}(s^{-1}(\alpha(p))\times\{p\}) = \{\mu(g,p)\mid g\in s^{-1}(\alpha(p))\}.
\]
Now observe that $i(i(g)) = g$; this follows for example from the uniqueness of the inverses (see \cite[Ex. 1.2, p. 9]{Meinrenken2017}, and  use the fact that the isotropy groups are Lie groups). Thus writing $h = i(g)$ we have that $t(h) = t(i(g)) = s(g) = \alpha(p)$ and 
\[
\{\mu(g,p)\mid g\in s^{-1}(\alpha(p))\} = \{\mu(i(h),p)\mid h\in t^{-1}(\alpha(p))\}.
\]
So we conclude that 
\[
\overline{t}^{-1}(p) = \{(g,\mu(i(g),p))\mid g\in t^{-1}(\alpha(p))\} = t^{-1}(\alpha(p))\times L_p.
\]
\end{proof}

With the previous description of the fibers of the submersion $\overline{t}\colon \G\times_M P\rightrightarrows P$ we get the following description for the $\overline{t}$-vertical space $\kernel D\overline{t}$.

\begin{cor}\th\label{C: Kernel of bar(t)}
Consider  $\G\times_M P\rightrightarrows P$ the action groupoid of a Lie groupoid  action of the proper Lie groupoid $\G\rightrightarrows M$ on $\alpha\colon P\to M$. Then
\[
\kernel D_{(g,p)}\overline{t} = \kernel D_gt \times T_{p}L_p,
\]
where $L_p\subset P$ is the orbit of the action groupoid $\G\times_M P\rightrightarrows P$ containing $p$.
\end{cor}

From this description of $\kernel D \overline{t}$, we deduce that $\bar{t}$ is a Riemannian submersion with respect to $\widehat{\eta}_\varepsilon$.

\begin{thm}[Proof of \th\ref{MT: deformation}]\th\label{T: contraction of groupoid induces cheeger deformation}
Consider $\G\times_M P\rightrightarrows P$ the action groupoid of a Lie groupoid  action $\mu$ of the proper Lie groupoid $\G\rightrightarrows M$ on $\alpha\colon P\to M$. Equip the metric $\widehat{\eta}_\varepsilon$ on $\G\times_M P$. Then $\overline{t}\colon (\G\times_M P,\widehat{\eta}_\varepsilon)\to P$ is a Riemannian submersion.
\end{thm}

\begin{proof}
We recall that for the partition $\fol_P = \{L_p\mid p\in P\}$ of $P$ given by the orbits of $\G\times_M P\rightrightarrows P$, the Riemannian metric $\eta^P$ gives a singular Riemannian foliation $(P,\eta^P,\fol_P)$.

We consider $(X,W)\in \kernel D_gt\times T_pL_p$. Also we consider the normal bundle $\nu(L_p)\to L_p$ of $L_p$ using the metric $\eta^{P}$, and denote by $(\eta^P)_p^\perp = \eta^P_p|_{T L_p^\perp}$ the Riemannian metric given by the restriction of $\eta^P$ to the normal spaces of the $\G \times_M P$-orbits.We now denote by $(\kernel Dt)^\perp$ the orthogonal complement of $\kernel Dt$ with respect to $\eta^{(1)}$, and by  $(\eta^{(1)})^\perp$ the restriction of $\eta^{(1)}$ to $(\kernel Dt)^\perp$. 

We observe that along $W\in TL_p$, the Riemannian metric $(\eta^P)^\perp$ is invariant: i.e. \linebreak$\Lie_W((\eta^P)^\perp) = 0$, since the $\G\times_M P$-orbits in $P$ are locally $\eta^P$-equidistant. Also we observe, that since $t\colon (\G,\eta^{(1)})\to M$ is a Riemannian submersion, we have $\Lie_X((\eta^{(1)})^\perp) =0$. By observing that for the $\widehat{\eta}_\varepsilon$-orthogonal complement of $(\kernel D\overline{t})$, we have $(\kernel D\overline{t})^\perp = (\kernel Dt)^\perp \times \nu(L_p)$, it follows that 
\[
\Lie_{(X,W)}(\widehat{\eta}_\varepsilon)^\perp = \frac{1}{\varepsilon}\Lie_X((\eta^{(1)})^\perp)+\Lie_W ((\eta^{(P)})^\perp) = 0.
\]
Thus we conclude by \cite[Theorem~1.2.1]{GromollWalschap} that $\overline{t}\colon (\G\times_M, \widehat{\eta}_\varepsilon)\to P$ is a Riemannian submersion.
\end{proof}

\begin{definition}\label{D: Cheeger deformation}
Given $\G\rightrightarrows M$ a proper Lie groupoid acting on \linebreak$\alpha\colon  P\to M$ via $\mu$, an $\eta^{(1)}$ a $1$-metric on $\G$ induced by some $2$-metric, and a $\mu$-transversaly invariant metric $\eta^P$ on $P$, we define the \emph{Cheeger deformation} of $\eta^P$ with respect to $\eta^{(1)}$ as the Riemannian metric $\eta_\varepsilon$ on $P$ induced by the Riemannian submersion $(\G\times_{M} P,\widehat{\eta}_\varepsilon)\to P$ as in \th\ref{T: source map for the action groupoid is a Riemannian submersion}.
\end{definition}

\subsection{Horizontal distribution} Given the Riemannian submersion $\overline{t}\colon (\G\times_M P,\eta^{(1)}+\eta^P)\to P$, we consider the vertical distribution $\V\colon \G\times_M P\to T ( \G\times_M P)$, given by $\V(g,p) = \kernel D_{(g,p)}\overline{t}$.

From now for $p\in P$ and  $X\in \kernel D_{\1_{\alpha(p)}}s\subset T_{\1_{\alpha(p)}}\G$, we denote by $X^\ast(p)\in T_p P$ the vector given by 
\[
	X^\ast(p) = -D_{(\1_{\alpha(p)},p)}\,\overline{t}(X,0).
\]

We define the \emph{Orbit tensor} $\Sh(p)\colon \kernel D_{\1_{\alpha(p)}}s\to \kernel D_{\1_{\alpha(p)}}s$ by the following relationship:
\begin{linenomath}
\begin{align*}
\eta^{(1)}(\Sh(p)(x),y) =& \eta^P(D_{(\1_{\alpha(p)},p)}\overline{t}(X,0),D_{(\1_{\alpha(p)},p)}\overline{t}(Y,0))\\
 =& \eta^P (X^\ast (p),Y^\ast (p)).
\end{align*}
\end{linenomath}
	
We give an alternative description of $\kernel D\overline{t}$, the vertical space of $\overline{t}$, that is similar to the one presented in \cite[Lemma~3.2]{Mueter}.

\begin{lemma}\th\label{L: Vertical space of target map of action Lie groupoid}
Let the proper Lie groupoid $\G\rightrightarrows M$ act effectively and smoothly on $\alpha\colon P\to M$. Then we have 
\[
	\V(\1_{\alpha(p)},p) = \{(x,X^\ast(p))\mid x\in \kernel D_{\1_{\alpha(p)}} s\}.
\]
\end{lemma}

\begin{proof}
By \th\ref{C: Kernel of bar(t)}, we have that $\V(\1_{\alpha(p)},p) = \kernel D t(\1_{\alpha(p)})\times T_pL_p$.

Recall that $L_p = \overline{t}(\overline{s}^{-1}(p)) = \{\overline{t}(g,p)\mid g\in s^{-1}(\alpha(p))\}$, thus we have that $TL_p = D\overline{t}(T\overline{s}^{-1}(p))$.  Thus we see that for any $X\in \kernel D_{\1_{\alpha(p)}} s = T_{\1_{\alpha(p)}}\overline{s}^{-1}(\alpha(p))$, that $X^\ast(p) = -D_{\1_{\alpha(p)}}\overline{t}(X,0)\in T_pL_p$.

We consider $\gamma\colon I\to s^{-1}(\alpha(p))$ a curve with $\gamma(0) = \1_{\alpha(p)}$ and $\gamma'(0) = X$ in \linebreak$\kernel Ds(\1_{\alpha(p)})$. Set $\widehat{\gamma}(t) = (\gamma(t),p)$. We point  that for any $t\in I$ $s(\gamma(t))= \alpha(p)$, and thus $\widehat{\gamma}(t) \in \G\times_M P$ for any $t\in I$. Moreover, $\widehat{\gamma}'(0) = (X,0)$. So the set $\{(x,X^\ast(p))\mid x\in \kernel Ds(\1_{\alpha(p)})\}$ spans $\kernel D_{(\1_{\alpha(p)},p)} \overline{t}$.
\end{proof}

For the Riemannian metric $\widehat{\eta}_\varepsilon$ on the action groupoid $\G \times_M P\rightrightarrows P$, we denote by $\Hor(g,p) = (\kernel D_{(g,p)}\,\overline{t})^\perp$ the horizontal space of the target map $\overline{t}\colon \G \times_M P\to P$, and proof the following description. 

\begin{thm}\th\label{T: bar(t)-Horizontal distribution}
Consider a proper Lie groupoid $\G\rightrightarrows M$ a Lie groupoid acting on $\alpha\colon P\to M$. Consider the metric $\widehat{\eta}_\varepsilon = \frac{1}{\varepsilon}\eta^{(1)}+\eta^{(P)}$ on the action groupoid $\G\times_M P\rightrightarrows P$. Then The $\overline{t}$-horizontal space with respect to $\widehat{\eta}_\varepsilon$ is given by:
\begin{linenomath}
\begin{align*}
\Hor_\varepsilon(\1_{\alpha(p)},p) = \Bigg\{\left(-\varepsilon\Sh(p)(x)+\frac{1}{\varepsilon}\zeta,X^\ast(p)+\xi\right)\mid& x\in \kernel D_{\1_{\alpha(p)}} s,\: \zeta\in (\kernel D_{\1_{\alpha(p)}} s)^\perp,\\
 &\mbox{ and } \xi\in  T_p(L_p)^\perp\Bigg\}.
\end{align*}
\end{linenomath}
\end{thm}

\begin{proof}
We consider a vector $(Y,X)\in T_{(\1_{\alpha(p)},p)}\, \G\times_M P$. We denote by $X^\top$ the component of $X$ tangent to the orbit $L_p\supset P$ of the action groupoid, and $X^\perp$ the complementary component of $X$ normal to $L_p$ with respect to $\eta^P$. In this way $X = X^\top+X^\perp$ in a unique way. But since $L_p = \overline{t}(\overline{s}^{-1}(p))$, then there exists $x\in \kernel D_{\1_{\alpha(p)}}s$, such that $X^\ast (p) = -D_{\1_{\alpha(p)}}\overline{t}(x,0) = X^\top$. And thus $X= X^\ast(p)+X^\perp$. 

The vector $(Y,X^\ast(p)+X^\perp)$ is $(\frac{1}{\varepsilon}\eta^{(1)}+\eta^P)$-horizontal if and only if for all $z\in \kernel D_{\1_{\alpha(p)}} s$, we have
\begin{linenomath}
\begin{align*}
0 &=  \left(\frac{1}{\varepsilon}\eta^{(1)}+\eta^P\right)\Bigl((Y,X^\ast(p)+X^\perp),(z,Z^\ast(p))\Bigr)\\
&= \frac{1}{\varepsilon}\eta^{(1)}(Y,z)+\eta^P(X^\ast(p),Z^\ast(p))\\
& = \frac{1}{\varepsilon}\eta^{(1)}(Y,z)+\eta^{(1)}(\Sh(p)(x),z)\\
&= \eta^{(1)}\left(\frac{1}{\varepsilon}Y+\Sh(p)(x),z\right).
\end{align*}
\end{linenomath}

We write $Y = Y^\top+Y^\perp$, where $Y^\top\in \kernel D_{\1_{\alpha(p)}}s$ and \linebreak$Y^\perp\in \bigl(\kernel D_{\1_{\alpha(p)}}s\bigr)^\perp$ with respect to $\eta^{(1)}$. Since $\Sh(p)(x)\in \kernel D_{\1_{\alpha(p)}}s$, we conclude that for all $z\in \kernel D_{\1_{\alpha(p)}}s$
\[
0=\eta^{(1)}\left(\frac{1}{\varepsilon}Y+\Sh(p)(x),z\right) = \eta^{(1)}\left(\frac{1}{\varepsilon}Y^\top+\Sh(p)(x),z\right).
\] 
Thus we conclude that $Y^\top = -t\Sh(p)(x)$, and thus the claim follows by setting $\xi = X^\perp$ and $\zeta = Y^\perp$.
\end{proof}

\begin{remark}
In contrast to \cite[Lemma~3.2 (b)]{Mueter}, we have an extra term $\zeta$. I the case when the foliation $\fol_M$ on $M$ induced by the orbits of $\G\rightrightarrows M$ consist of only one leaf, then 
\begin{linenomath}
\begin{align*}
\Hor(\1_{\alpha(p)},p) = \Big\{\big(-\varepsilon\Sh(p)(x),X^\ast(p)+\xi\big)\mid&\, x\in \kernel Ds(\1_{\alpha(p)}),\\
&\mbox{ and }\xi\in \nu_p(L_p)\Big\},
\end{align*}
\end{linenomath}
and thus we recover \cite[Lemma~3.2 (b)]{Mueter} for Lie  group actions.
\end{remark}

\subsection{Cheeger tensor} We define the \emph{Cheeger tensor of $\eta^P$ along the action of $\G\rightrightarrows M$ on $\alpha\colon P\to M$} for the metric $\eta_\varepsilon$ induced on $P$ by  \th\ref{T: contraction of groupoid induces cheeger deformation} to be the map $\Ch^{-1}_\varepsilon(p)\colon T_p P\to T_pP$ defined as follows: We write $X\in T_pP$ as $X = X^\top+X^\perp$, where $X^\top \in T_pL_p$ and $X^\perp \in \nu_p(L_p)$. Thus, there exists an $x\in \kernel D_{\1_{\alpha(p)}}\, s$ such that $X^\ast(p) = X^\top$. With this decomposition $X = X^\ast(p)+X^\perp$ we set
\[
\Ch^{-1}_\varepsilon(p)(X^\ast(p)+X^\perp ) = (\Id+\varepsilon\Sh(p)(x))^\ast(p)+X^\perp
\]

\begin{lemma}\th\label{L: Cheeger tensor does not depend of choice of x such that xast is X}
Let $\G\times_M P\rightrightarrows P$ be the groupoid action of a Lie groupoid action of the proper Lie groupoid $\G\rightrightarrows M$ on $\alpha\colon P\to M$. For $X\in TP$, such that $X(p)\in T_p L_p$ for all $p\in P$, the Cheeger tensor $\Ch^{-1}_\varepsilon(p)$ does not depend on the choice of $x\in \kernel D_{\1_{\alpha(p)}}\, s$ such that $-D_{\1_{\alpha(p)}}\overline{t}(x,0) = X(p)$. 
\end{lemma}

\begin{proof}
We take $x,y\in  \kernel D_{\1_{\alpha(p)}} s$ with 
\[-D_{(\1_{\alpha(p)},p)}\overline{t}(x,0) = X(p) = -D_{(\1_{\alpha(p)},p)}\overline{t}(y,0)
\] for all $p\in P$. This is equivalent to  $x-y\in \kernel D_{(\1_{\alpha(p)},p)}\,\overline{t}$. Moreover since $L_p =\overline{t}(\overline{s}^{-1}(p)) = \overline{t}(s^{-1}(\alpha(p))\times\{p\})$, we have $T_p L_p = D_{(\1_{\alpha(p)},p)}\,\overline{t}(\kernel D_{\1_{\alpha(p)}} s\times \{0\})$. This implies that $x-y\in \kernel D_{\1_{\alpha(p)}} s$. Then we get that for any $z\in \kernel D_{\1_{\alpha(p)}} s$ the following
\[
\eta^{(1)}(\Sh(p)(x-y),z) = \eta^{(P)}(-D_{(\1_{\alpha(p)},p)}\overline{t}(x-y,0), Z^\ast(p)) = 0.
\]
This implies that $\Sh(p)(x) = \Sh(p)(y)$. Thus we obtain that 
\begin{linenomath}
\begin{align*}
-D_{(\1_{\alpha(p)},p)}\overline{t}(y+\varepsilon\Sh(p)(y),0) =&  -D_{(\1_{\alpha(p)},p)}\overline{t}(y+\varepsilon\Sh(p)(y),0)- D_{(\1_{\alpha(p)},p)}\overline{t}(x-y,0)\\
=& -D_{(\1_{\alpha(p)},p)}\overline{t}(y+x-y+\varepsilon\Sh(p)(y),0)\\
=& -D_{(\1_{\alpha(p)},p)}\overline{t}(x+\varepsilon\Sh(p)(y),0)\\
=& -D_{(\1_{\alpha(p)},p)}\overline{t}(x+\varepsilon\Sh(p)(x),0).
\end{align*}
\end{linenomath}
\end{proof}

We define the \emph{horizontal lift at $p\in P$ of tangent vectors of $P$} as the map $h(p)\colon T_p P \to T_{(\1_{\alpha(p)},p)}\G\times_M P$ given by
\begin{equation}
h(p)(X) = h(p)(X^\ast(p)+X^\perp) = (-\varepsilon\Sh(p)(x), X).\tag{\mbox{Horizontal}}\label{E: Horizontal lifts of TGXP->P at 1(P)}
\end{equation}
Here we are using the fact that for $X^\top\in T_p(L_p)$ there exists an $x\in \kernel D_{\1_{\alpha(p)}}\, s$ such that $X^\ast(p) = X^\top$. We observe that it follows from the proof of \th\ref{L: Cheeger tensor does not depend of choice of x such that xast is X}, that $\Sh(p)(x)$ does not depend of the choice of $x\in \kernel D_{\1_{\alpha(p)}} s$, such that $X^\ast(p) = X^\top$. In other words, the definition of $h$ does not depend on the choice of $x\in \kernel D_{\1_{\alpha(p)}} s$.

\begin{remark}
From the description of $\Hor(\1_{\alpha(p)},p)$ in \th\ref{T: bar(t)-Horizontal distribution} we conclude that $h(p)(X)\in \Hor(\1_{\alpha(p)},p)$.
\end{remark}

\begin{lemma}\th\label{L: target map of groupoid action is the identity on the normal part of the orbits}
Consider a proper Lie groupoid  $\G\rightrightarrows M$ acting  on $\alpha\colon P\to M$ via $\mu\colon \G\times_M P \to P$. Let $\eta^{(1)}$ be a $1$-metric  on $\G$, and consider $\eta^P$ a $\G$-invariant Riemannian metric on $P$. Moreover, assume that $\nu_p(L_p)\subset T_p\, \alpha^{-1}(\alpha(p))$. Then for $\xi\in \nu_p (L_p)$ we have that 
\[
D_{(\1_{\alpha(p)},p)}\overline{t}(0,\xi) = \xi.
\]
\end{lemma}

\begin{proof}
Consider any curve  $\gamma\colon I\to \alpha^{-1}(\alpha(p))$   with $\gamma(0) = p$ and $\gamma'(0)= \xi$. Then the curve $\widehat{\gamma}\colon I\to \G\times_M P$ given by $\widehat{\gamma}(r) = \big(\1_{\alpha(p)},\gamma(r)\big)$ is a curve in $\G\times_M P$, since $s(\1_{\alpha(p)}) =\alpha(p)= \alpha(\gamma(r))$, and $\widehat{\gamma}(0) = (\1_{\alpha(p)},p)$. Moreover, we have $\widehat{\gamma}'(0) =(0,\xi)$.

Now we simply evaluate $\overline{t}$ along $\widehat{\gamma}$:
\begin{linenomath}
\begin{align*}
\overline{t}(\widehat{\gamma}(r)) &= \mu(\1_{\alpha(p)},\gamma(r))\\
&= \mu(\1_{\alpha(\gamma(\ell))},\gamma(\ell))\\
&= \gamma(\ell).
\end{align*}
Thus we conclude that 
\[
D_{(\1_{\alpha(p)},p)}\overline{t}(0,\xi) = \frac{d}{dr} \overline{t}(\1_{\alpha(p)},\gamma(r))|_{r=0}= \frac{d}{dr} \gamma(r)|_{r=0} = \xi.
\]
Since $\nu_p (L_p)\subset T_p\, \alpha^{-1}(\alpha(p))$, for $\xi\in \nu_p (L_p)$ we may find a curve as above. This proves the result.
\end{linenomath}
\end{proof}

\begin{remark}
We point out that for a Lie group action $\mu$ of  $G\rightrightarrows\{\ast\}$ along $\alpha\colon M\to\{\ast\}$, we automatically have that $\nu_p L_p \subset T_p M = T_p \alpha^{-1}(\ast) = T_p \alpha^{-1}(\alpha(p))$. 
\end{remark}

\begin{remark}
We consider an orbit-like foliation $(M,\fol)$ and fix a leaf $L_p$. Then for some saturated tubular neighborhood $U$ of $L$ in $M$, by \th\ref{T: Local transversal homogeneous subfoliation} there exists a Lie groupoid $\G\rightrightarrows L$, and a Lie groupoid action $\mu$ of $\G\rightrightarrows L$ on $\pi\colon \nu(L)\to L$, such that the foliation $\fol|_{U}$ is given by the orbits of the groupoid action, when we identify $U$ with a tubular neighborhood $\nu^\delta (L)$ of the $0$-section in $\nu(L)$ via the normal exponential map. In this case we have  for $\xi\in \nu^\delta(L)$, that
\[
\nu_{\xi}(L_\xi)\subset  T_\xi(\nu^\delta(L)) \cong \nu^\delta_{\pi(\xi)}(L)\cong T_\xi(\nu^\delta_{\pi(\xi)}(L)) = T_\xi \Big((\pi)^{-1}(\pi(\xi))\Big).
\]
\end{remark}

We now proceed to show that the horizontal vector $h(p)$ defined by \eqref{E: Horizontal lifts of TGXP->P at 1(P)} is the horizontal lift of the Cheeger thensor $\Ch^{-1}_\varepsilon(p)$.

\begin{thm}\th\label{T: h is horizontal lift of cheeger}
Consider $\G\rightrightarrows M$ a proper Lie groupoid acting on $\alpha\colon P\to M$ via $\mu$, and $\eta^{(1)}$ a $1$-metric  on $\G$. Let $\eta^P$ an $\mu$-transversely invariant Riemannian metric on $P$, and assume that for all $p\in M$ we have $\nu_p (L_p)\subset T_p \alpha^{-1}(\alpha(p))$. We consider $-D\overline{t}\colon T(\G\times_M P)\to TP$. Then for all $X\in T_p P$, the vector $h(p)(X)$ is an $D\overline{t}$-horizontal lift of $Ch^{-1}_\varepsilon(p)(X)$, i.e.
\[
D\overline{t}(h(p)(X)) = Ch^{-1}_\varepsilon(p)(X).
\]
\end{thm}

\begin{proof}
We consider $X\in T_p P$, and assume that $X = X^\ast(p)+X^\perp$ for some $x\in \kernel D_{\1_{\alpha(p)}} s$. Then we have 
\begin{linenomath}
\begin{align*}
\Ch^{-1}_{\varepsilon}(p) (X) &= -D_{(\1_{\alpha(p)},p)} \overline{t}(x+\varepsilon \Sh(p)(x),0)+X^\perp\\
&=D_{(\1_{\alpha(p)},p)} \overline{t}(-x-\varepsilon \Sh(p)(x),0)+X^\perp\\
& = D_{(\1_{\alpha(p)},p)} \overline{t}(-x-\varepsilon \Sh(p)(x),X^\perp),
\end{align*}
\end{linenomath}
since by our hypothesis $\nu_p(L_p)\subset T_p\, \alpha^{-1}(\alpha(p))$ and \th\ref{L: target map of groupoid action is the identity on the normal part of the orbits}, we have \linebreak$D_{(\1_{\alpha(p)},p)}\overline{t}(0,X^\perp) = X^\perp$. Now we recall that by \th\ref{L: Vertical space of target map of action Lie groupoid}, we have that 
\[
D_{(\1_{\alpha(p)},p)} \overline{t}(x,X^\ast(p)) =0.
\]
Thus we conclude that 
\begin{linenomath}
\begin{align*}
\Ch^{-1}_{\varepsilon}(p) (X) &=D_{(\1_{\alpha(p)},p)} \overline{t}(-x-\varepsilon \Sh(p)(x),X^\perp) + D_{(\1_{\alpha(p)},p)} \overline{t}(x,X^\ast(p))\\
&= D_{(\1_{\alpha(p)},p)} \overline{t}(-x-\varepsilon \Sh(p)(x)+x,X^\ast(p)+X^\perp)\\
&= D_{(\1_{\alpha(p)},p)} \overline{t}(-\varepsilon\Sh(p)(x), X)\\
& =D_{(\1_{\alpha(p)},p)} \overline{t}\big( h(p)(X)\big).
\end{align*}
\end{linenomath}
\end{proof}

We can describe the metric $\eta_\varepsilon$ on $P$ explicitly in terms of our original metric $\eta^P$ on $P$ and the Cheeger tensor.

\begin{thm}\th\label{T: cheeger deformation controlled by cheeger tensor}
Consider $\G\rightrightarrows M$ a proper Lie groupoid acting on $\alpha\colon P\to M$  via $\mu$, and $\eta^{(1)}$  a $1$-metric on  $\G$. Let $\eta^P$ be a $\mu$-transversely invariant Riemannian metric on $P$, and assume that for all $p\in M$ we have $\nu_p (L_p)\subset T_p \alpha^{-1}(\alpha(p))$. Then for $X,Y\in T_pP$ we have
\[
\eta_\varepsilon(X,Y) = \eta^P(Ch_\varepsilon(p)(X),Y).
\]
\end{thm}

\begin{proof}
By \th\ref{T: h is horizontal lift of cheeger}, writing $X=X^\ast(p)+X^\perp$ and $Y= Y^\ast(p)+Y^\perp$ for some $x,y\in \kernel D_{\1_{\alpha(p)}}s$, and using that \[
D_{(\1_{\alpha(p)},p)}\overline{t}\colon (\Hor(\1_{\alpha(p)},p),\widehat{\eta}_\varepsilon) \to (T_p P,\eta_\varepsilon)
\]
is an isometry, we have that 
\begin{linenomath}
\begin{align*}
\eta_\varepsilon(\Ch^{-1}_\varepsilon(p)(X),\Ch^{-1}_\varepsilon(p)(Y)) &= \eta_\varepsilon\Big(D_{(\1_{\alpha(p)},p)}\overline{t}\big(h(p)(X)\big),D_{(\1_{\alpha(p)},p)}\overline{t}\big(h(p)(Y)\big)\Big)\\
&= \widehat{\eta}_\varepsilon \big(h(p)(X),h(p)(Y)\big)\\
&= \left(\frac{1}{\varepsilon}\eta^{(1)}+\eta^{(P)}\right)\Big(\big(-\varepsilon\Sh(p)(x),X\big),\big(-\varepsilon\Sh(p)(y),Y\big)\Big)\\
& = \left(\frac{1}{\varepsilon}\eta^{(1)}\Big(\varepsilon\Sh(p)(x),\varepsilon\Sh(p)(y)\Big)\right)+\eta^P(X,Y)\\
& = \eta^{(1)}\Big(\Sh(p)(x),\varepsilon\Sh(p)(y)\Big)+\eta^P (X,Y)\\
&= \eta^P\Big(X,\varepsilon(\Sh(p)(y))^\ast(p)\Big)+\eta^P(X,Y)\\
& = \eta^P\Big(X,\varepsilon(\Sh(p)(y))^\ast(p)+Y^\ast(p)+Y^\perp\Big)\\
& = \eta^P\Big(X,(y+\varepsilon(\Sh(p)(y))^\ast(p)+Y^\perp\Big)\\
&= \eta^P\big(X,\Ch^{-1}_\varepsilon(p)(Y)\big).
\end{align*}
\end{linenomath}
Now taking $\tilde{X}= \Ch_\varepsilon(p)(X)$ and $\tilde{Y} = \Ch_\varepsilon(p)(Y)$ we obtain
\begin{linenomath}
\begin{align*}
\eta_\varepsilon(X,Y) &= \eta_\varepsilon(\Ch^{-1}_\varepsilon(p)(\tilde{X}),\Ch^{-1}_\varepsilon(p)(\tilde{Y}))\\
&= \eta^P(\tilde{X},\Ch^{-1}_\varepsilon(p)(\tilde{Y}))\\
& = \eta^P(\Ch_\varepsilon(p)(X),Y).
\end{align*}
\end{linenomath}
\end{proof}

With this we see that $\eta_\varepsilon$ is a deformation of our original $\mu$-transversely invariant metric $\eta^{P}$ on $P$,

\begin{thm}\th\label{T: Cheeger deformation collapses}
Consider $\G\rightrightarrows M$ a proper Lie groupoid acting on $\alpha\colon P\to M$ via $\mu$, and $\eta^{(1)}$ a $1$-metric on $\G$. Let $\eta^P$ be a $\mu$-transversely invariant Riemannian metric on $P$, and assume that for all $p\in M$ we have $\nu_p (L_p)\subset T_p \alpha^{-1}(\alpha(p))$. Then setting $\fol =\{L_p\mid p\in P\}$, we have that $(P,d_{\eta_{\varepsilon}})\to (P/\fol, d_{\eta^P}^\ast)$ in the Gromov-Hausdorff topology as $\varepsilon\to \infty$, and $\eta_{\varepsilon}\to \eta^P$ in the $C^{\infty}$-topology as $\varepsilon\to 0$. 
\end{thm}

\begin{proof}
We have that as $\varepsilon\to \infty$, then the space $(\G\times_M P,\frac{1}{\varepsilon}\eta^{(1)}+\eta^P)$ converges in the Gromov-Hausdorff sense to $(P,\eta^P)$. In particular the length of  vectors in $\kernel_{\1_{\alpha(p)}} s\times \{0\}$ goes to  $0$ as $\varepsilon \to \infty$. But since this vectors span the tangent spaces of the ($\G\times_M P\rightrightarrows P)$-orbits in $P$, then we conclude that the length of the vectors tangent to the orbits with respect to $\eta_\varepsilon$ goes to $0$. Now observe that for a vector $\xi\in T_pP$ perpendicular to the orbit $L_p$, we have by \th\ref{L: target map of groupoid action is the identity on the normal part of the orbits} and \th\ref{T: cheeger deformation controlled by cheeger tensor} that
\[
\eta_\varepsilon(\xi,\xi) = \eta^{P}\Big(\Ch_\varepsilon(p)(\xi),\xi\Big) = \eta^{P}(\xi,\xi).
\]
This implies that the foliation $\fol = \{L_p\mid L_p \mbox{ is } (\G\times_M P\rightrightarrows P)-\mbox{orbit}\}$ is a singular Riemannian foliation with respect to $\eta_\varepsilon$. Thus the distance between the leaves of $\fol$ is constant with respect to $\varepsilon$. Thus we conclude that $(P,\eta^\varepsilon)$ converges in the Gromov-Hausdorff sense to $(P/\fol,d_{\eta^{P}}^\ast)$, where $d_{\eta^{P}}^\ast$ is the metric induced by the Riemannian distance $d_{\eta^{P}}$ on $P$.

Observe that the tensor $\Ch^{-1}_\varepsilon(p)$ converges in the $C^\infty$-topology to the identity as $\varepsilon\to 0$. Thus the second claim follows.
\end{proof}

\subsection{Curvature} In this section we compute the sectional curvature of the deformed Riemannian manifold $(P,\eta_\varepsilon)$. We can use  O'Neil's formula (\th\ref{T: ONeills formula}) to compute the sectional curvature of our new deformed metrics. 

\begin{thm}\th\label{T: Proof of theorem B}
Consider $\G\rightrightarrows M$ a proper Lie groupoid acting on $\alpha\colon P\to M$ a submersion, and $\eta^{(1)}$ a Riemannian metric on $\G$ making it a Riemannian Lie groupoid. Let $\eta^P$ a left-$\G$-invariant Riemannian metric on $P$, and assume that for all $p\in M$ we have $\nu_p(L_p)\subset T_p \alpha^{-1}(\alpha(p))$. Then for the Cheeger deformation $(P,\eta_\varepsilon)$ we have for $v,w\in T_pP$ linearly independent vectors, with $v = X^\ast(p)+v^\perp$ and $w = Y^\ast(p)+w^\perp$ for some $x,y\in \kernel D_{\1_{\alpha(p)}} s$. Then
\begin{linenomath}
\begin{align}\label{EQ: Full curvature description}
\begin{split}
K_{\eta_\varepsilon}\Big(\Ch^{-1}_\varepsilon(p)&(v),\Ch^{-1}_\varepsilon(p)(w)\Big)  = K_{\eta^P}(v,w)+\varepsilon^3K_{\eta^{(1)}}\Big(\Sh(p)(v),\Sh(p)(w)\Big)\\
+&3\Big\|A_{h(p)(v)}h(p)(w)\Big\|^2_{\left(\frac{1}{\varepsilon}\eta^{(1)}+\eta^P\right)}\\
+&\Big\|\II_\varepsilon\big((v,-\varepsilon \Sh(p)(v)),(w,-\varepsilon\Sh(p)(w))\big)\Big\|^2_{\frac{1}{\varepsilon}\eta^{(1)}+\eta^P}\\
 -&\left(\frac{1}{\varepsilon}\eta^{(1)}+\eta^P\right)\Big(\II_\varepsilon\big((v,-\varepsilon\Sh(p)(v)),(v,-\varepsilon\Sh(p)(v))\big),\\
 &\II_\varepsilon\big((w,-\varepsilon\Sh(p)(w)),(w,-\varepsilon\Sh(p)(w))\big)\Big).
\end{split}
\end{align}
\end{linenomath}
where $\II(\cdot,\cdot)$ is the second fundamental form of $(\G\times_M P,\widehat{\eta}_\varepsilon)\subset (\G\times P,(1/\varepsilon)\eta^{(1)}+\eta^{P})$, and $A_\cdot \cdot$ is the $A$-tensor of the Riemannian submersion $\bar{t}\colon (\G\times_M P,\hat{\eta}_\varepsilon)\to (P,\eta_t)$.
\end{thm}

\begin{proof}
Observe that since $v$ and $w$ are linearly independent, so are $\Ch^{-1}(p)(v)$ and $\Ch^{-1}(p)(w)$. Taking $\widehat{\eta}_\varepsilon = \frac{1}{\varepsilon}\eta^{(1)}+\eta^{P}$, from \th\ref{T: h is horizontal lift of cheeger}, and \th\ref{T: ONeills formula} we have that 
\begin{linenomath}
\begin{align*}
K_{\eta_\varepsilon}(\Ch^{-1}_\varepsilon(p)(v),\Ch^{-1}_\varepsilon(p)(w)) =& K_{\widehat{\eta}_\varepsilon}(h(p)(v),h(p)(w))\\ 
&+ 3\|A_{h(p)(v)} h(p)(w)\|^2_{\widehat{\eta}_\varepsilon}.
\end{align*}
\end{linenomath}
We now apply the Gauss-formula to compute that
\begin{align*}
K_{\widehat{\eta}_\varepsilon}\big(h(p)(v),h(p)(w)\big)=&K_{\widehat{\eta}_\varepsilon}\big((-\varepsilon \Sh(p)(v),v),(-\varepsilon \Sh(p)(w),w)\big)\\
=&\varepsilon^3 K_{\eta^{(1)}}\big(\Sh(p)(v),\Sh(p)(w)\big)+K_{\eta^{P}}(v,w)\\
&+\left(\frac{1}{\varepsilon}\eta^{(1)}+\eta^{P}\right)\bigg(\II_\varepsilon\big((-\varepsilon\Sh(p)(v),v\big),(-\varepsilon\Sh(p)(v),v)\big),\\
&\;\II_\varepsilon\big((-\varepsilon\Sh(p)(w),w),(-\varepsilon\Sh(p)(w),w)\big)\bigg)\\
&-\left(\frac{1}{\varepsilon}\eta^{(1)}+\eta^{P}\right)\bigg(\II_\varepsilon\big((-\varepsilon\Sh(p)(v),v\big),(-\varepsilon\Sh(p)(w),w)\big),\\
&\;\II_\varepsilon\big((-\varepsilon\Sh(p)(v),v),(-\varepsilon\Sh(p)(w),w)\big)\bigg)
\end{align*}
Combining these two expressions we obtain the desired result.
\end{proof}

\bibliographystyle{siam}
\bibliography{Bibliography.bib}

\begin{thebibliography}{10}

\bibitem{Alexandrino2010}
{\sc M.~M. Alexandrino}, {\em Desingularization of singular {R}iemannian
  foliation}, Geom. Dedicata, 149 (2010), pp.~397--416.

\bibitem{Alexandrino}
{\sc M.~M. Alexandrino and R.~G. Bettiol}, {\em Lie groups and geometric
  aspects of isometric actions}, Springer, Cham, 2015.

\bibitem{AlexandrinoInagakiStruchiner2018}
{\sc M.~M. Alexandrino, M.~K. Inagaki, and I.~Struchiner}, {\em Lie groupoids
  and semi-local models of singular riemannian foliations}, Ann. Glob. Anal.
  Geom., 61 (2022), pp.~593--619.

\bibitem{AlexandrinoRadeschi2016}
{\sc M.~M. Alexandrino and M.~Radeschi}, {\em Mean curvature flow of singular
  {R}iemannian foliations}, J. Geom. Anal., 26 (2016), pp.~2204--2220.

\bibitem{AndroulidakisSkandalis2009}
{\sc I.~Androulidakis and G.~Skandalis}, {\em The holonomy groupoid of a
  singular foliation}, J. Reine Angew. Math., 626 (2009), pp.~1--37.

\bibitem{CavenaghiESilvaSperanca2023}
{\sc L.~Cavenaghi, R.~e~Silva, and L.~Speran{\c{c}}a}, {\em Positive ricci
  curvature through cheeger deformations}, Collect. Math,  (2023).

\bibitem{Cheeger1973}
{\sc J.~Cheeger}, {\em Some examples of manifolds of nonnegative curvature}, J.
  Differential Geometry, 8 (1973), pp.~623--628.

\bibitem{CheegerEbin1975}
{\sc J.~Cheeger and D.~G. Ebin}, {\em Comparison theorems in {R}iemannian
  geometry}, North-Holland Mathematical Library, Vol. 9, North-Holland
  Publishing Co., Amsterdam-Oxford; American Elsevier Publishing Co., Inc., New
  York, 1975.

\bibitem{Corro}
{\sc D.~Corro}, {\em Manifolds with aspherical singular {R}iemannian
  foliations}, PhD thesis, Karlsruhe Institute of Technology, 2018.
\newblock \url{https://publikationen.bibliothek.kit.edu/1000085363}.

\bibitem{Corro2019}
\leavevmode\vrule height 2pt depth -1.6pt width 23pt, {\em A-foliations of
  codimension two on compact simply-connected manifolds}, Math. Z., 304 (2023),
  pp.~Paper No. 63, 32.

\bibitem{CorroMoreno2020}
{\sc D.~Corro and A.~Moreno}, {\em Core reduction for singular {R}iemannian
  foliations and applications to positive curvature}, Ann. Global Anal. Geom.,
  62 (2022), pp.~617--634.

\bibitem{Debord2001}
{\sc C.~Debord}, {\em Holonomy groupoids of singular foliations}, J.
  Differential Geom., 58 (2001), pp.~467--500.

\bibitem{delHoyoFernandes2018}
{\sc M.~del Hoyo and R.~L. Fernandes}, {\em Riemannian metrics on {L}ie
  groupoids}, J. Reine Angew. Math., 735 (2018), pp.~143--173.

\bibitem{delHoyo2013}
{\sc M.~L. del Hoyo}, {\em Lie groupoids and their orbispaces}, Portugal.
  Math., 70 (2013), pp.~161--209.

\bibitem{Ehresmann1965}
{\sc C.~Ehresmann}, {\em Cat\'egories et structures}, Dunod, Paris, 1965.

\bibitem{GalazGarcia2015}
{\sc F.~Galaz-Garcia and M.~Radeschi}, {\em Singular {R}iemannian foliations
  and applications to positive and non-negative curvature}, J. Topol., 8
  (2015), pp.~603--620.

\bibitem{GallegoGualdraniHectorReventos1989}
{\sc E.~Gallego, L.~Gualandri, G.~Hector, and A.~Revent\'{o}s}, {\em
  Groupo\"{\i}des riemanniens}, Publ. Mat., 33 (1989), pp.~417--422.

\bibitem{GarmendiaGonzales2019}
{\sc A.~G. Garmendia~Gonz\'{a}lez}, {\em Groupoids and singular foliations},
  PhD thesis, Katholieke Universiteit Leuven, 2019.

\bibitem{GeRadeschi2013}
{\sc J.~Ge and M.~Radeschi}, {\em Differentiable classification of 4-manifolds
  with singular {R}iemannian foliations}, Math. Ann., 363 (2015), pp.~525--548.

\bibitem{Glickenstein2008}
{\sc D.~Glickenstein}, {\em Riemannian groupoids and solitons for
  three-dimensional homogeneous {R}icci and cross-curvature flows}, Int. Math.
  Res. Not. IMRN,  (2008), pp.~Art. ID rnn034, 49.

\bibitem{GromollWalschap}
{\sc D.~Gromoll and G.~Walschap}, {\em Metric foliations and curvature},
  vol.~268 of Progress in Mathematics, Birkh\"auser Verlag, Basel, 2009.

\bibitem{GroveZiller2002}
{\sc K.~Grove and W.~Ziller}, {\em Cohomogeneity one manifolds with positive
  {R}icci curvature}, Invent. Math., 149 (2002), pp.~619--646.

\bibitem{LawsonYau1972}
{\sc H.~B. Lawson, Jr. and S.~T. Yau}, {\em Compact manifolds of nonpositive
  curvature}, J. Differential Geometry, 7 (1972), pp.~211--228.

\bibitem{Lee}
{\sc J.~M. Lee}, {\em Introduction to smooth manifolds}, vol.~218 of Graduate
  Texts in Mathematics, Springer, New York, second~ed., 2013.

\bibitem{Meinrenken2017}
{\sc E.~Meinrenken}, {\em Lie groupoids and {L}ie algebroids lecture notes,
  {F}all 2017}, 2017.
\newblock
  \url{https://www.math.toronto.edu/mein/teaching/MAT1341_LieGroupoids/Groupoids.pdf},
  Last accessed February 2023.

\bibitem{MendesRadeschi2015}
{\sc R.~A.~E. Mendes and M.~Radeschi}, {\em A slice theorem for singular
  {R}iemannian foliations, with applications}, Trans. Amer. Math. Soc., 371
  (2019), pp.~4931--4949.

\bibitem{Moerdijk}
{\sc I.~Moerdijk and J.~Mr\v{c}un}, {\em Introduction to foliations and {L}ie
  groupoids}, vol.~91 of Cambridge Studies in Advanced Mathematics, Cambridge
  University Press, Cambridge, 2003.

\bibitem{Molino}
{\sc P.~Molino}, {\em Riemannian foliations}, vol.~73 of Progress in
  Mathematics, Birkh\"auser Boston, Inc., Boston, MA, 1988.

\bibitem{Moreno2019}
{\sc A.~Moreno}, {\em Point leaf maximal singular {R}iemannian foliations in
  positive curvature}, Differential Geom. Appl., 66 (2019), pp.~181--195.

\bibitem{MoullieWebpage}
{\sc L.~Moulli\'{e}}, {\em Cajun geometer-cheeger deformations}, 2017.
\newblock Last accessed 14 June 2023
  \url{https://lawrencemouille.wordpress.com/2017/02/27/cheeger-deformations/}.

\bibitem{Mueter}
{\sc M.~M\"{u}ter}, {\em Kr\"{u}mmungseh\"{o}hende {D}eformationen mittels
  {G}ruppenaktionen}, PhD thesis, Westf\"{a}lischen
  {W}ilhelms-{U}niversit\"{a}t {M}\"{u}nster, 1987.

\bibitem{ONeill1966}
{\sc B.~O'Neill}, {\em The fundamental equations of a submersion}, Michigan
  Math. J., 13 (1966), pp.~459--469.

\bibitem{Radeschi2014}
{\sc M.~Radeschi}, {\em Clifford algebras and new singular {R}iemannian
  foliations in spheres}, Geom. Funct. Anal., 24 (2014), pp.~1660--1682.

\bibitem{Radeschi-notes}
\leavevmode\vrule height 2pt depth -1.6pt width 23pt, {\em {Lecture notes on
  singular {R}iemannian foliations}}, 2017.
\newblock URL:
  \url{https://static1.squarespace.com/static/5994498937c5815907f7eb12/t/5998477717bffc656afd46e0/1503151996268/SRF+Lecture+Notes.pdf}.
  Last visited on May 2022.

\bibitem{Searle2023}
{\sc C.~Searle}, {\em Symmetries of spaces with lower curvature bounds},
  Notices Amer. Math. Soc., 70 (2023), pp.~564--575.

\bibitem{SearleSolorzanoWilhelm2015}
{\sc C.~Searle, P.~Sol\'{o}rzano, and F.~Wilhelm}, {\em Regularization via
  {C}heeger deformations}, Ann. Global Anal. Geom., 48 (2015), pp.~295--303.

\bibitem{SearleWilhelm2015}
{\sc C.~Searle and F.~Wilhelm}, {\em How to lift positive {R}icci curvature},
  Geom. Topol., 19 (2015), pp.~1409--1475.

\bibitem{Wang2018}
{\sc K.~J.~L. Wang}, {\em Proper Lie groupoids and their orbit spaces}, PhD
  thesis, Universiteit van Amsterdam, 2018.

\bibitem{Weber2009}
{\sc B.~Weber}, {\em Lecture notes ``{C}ollapsing and {F}-{S}tructures'' --
  lecture 3 ``the basic examples of collapse''}, 2009.
\newblock
  \url{https://www2.math.upenn.edu/~brweber/Lectures/USTC2011/Lecture%204%20-%20Examples%20of%20Collapse.pdf},
  consulted January 2023.

\bibitem{Winkelnkemper1983}
{\sc H.~E. Winkelnkemper}, {\em The graph of a foliation}, Ann. Global Anal.
  Geom., 1 (1983), pp.~51--75.

\bibitem{Yamabe1950}
{\sc H.~Yamabe}, {\em On an arcwise connected subgroup of a {L}ie group}, Osaka
  Math. J., 2 (1950), pp.~13--14.

\end{thebibliography}
\end{document}